\documentclass[12pt]{article}
\usepackage[french]{babel}
\usepackage{graphicx}
\usepackage{amsmath}
\usepackage{amsfonts}
\usepackage{amssymb}
\usepackage{amsthm}
\usepackage{hyperref}
\usepackage[T1]{fontenc}
\usepackage{color}
\usepackage{mathtools}
\usepackage{comment}

\input xy
\xyoption{all}

\newtheorem{thm}{Th\'eor\`eme}[section]
\newtheorem{cor}[thm]{Corollaire}
\newtheorem{lem}[thm]{Lemme}
\newtheorem{pro}[thm]{Proposition}

\theoremstyle{definition}
\newtheorem{defi}[thm]{D\'efinition}

\newtheoremstyle{remarque}{}{}{}{}{\it}{.}{\newline}{}
\theoremstyle{remarque}
\newtheorem*{rem}{Remarque}
\newtheorem*{rems}{Remarques}

\newcommand{\asd}[5]{%
\setbox1=\hbox{\ensuremath{^{#1}}}%
\setbox2=\hbox{\ensuremath{_{#2}}}%
\setbox5=\hbox{\ensuremath{#5}}%
\hspace{\ifnum\wd1>\wd2\wd1\else\wd2\fi}%
\ensuremath{\copy5^{\hspace{-\wd1}\hspace{-\wd5}#1\hspace{\wd5}#3}%
_{\hspace{-\wd2}\hspace{-\wd5}#2\hspace{\wd5}#4}%
}}

\renewcommand{\cal}[1]{\mathcal{#1}}

\usepackage[OT2,T1]{fontenc}
\DeclareSymbolFont{cyrletters}{OT2}{wncyr}{m}{n}
\DeclareMathSymbol{\Sha}{\mathalpha}{cyrletters}{"58}
\DeclareMathSymbol{\Brusse}{\mathalpha}{cyrletters}{"42}

\newcommand{\n}{\mathbb{N}}
\newcommand{\z}{\mathbb{Z}}
\newcommand{\q}{\mathbb{Q}}

\newcommand{\id}{\mathrm{id}}

\renewcommand{\hom}{\mathrm{Hom}}
\newcommand{\ext}{\mathrm{Ext}}
\newcommand{\aut}{\mathrm{Aut}}
\newcommand{\saut}{\mathrm{SAut}}
\newcommand{\out}{\mathrm{Out}}
\newcommand{\sout}{\mathrm{SOut}}
\renewcommand{\int}{\mathrm{Int}\,}

\newcommand{\gal}{\mathrm{Gal}}
\newcommand{\Ger}{\mathrm{Ger}}

\newcommand{\TORS}{\underline{\mathrm{TORS}}}
\newcommand{\EspHom}{\underline{\mathrm{EspHom}}}
\newcommand{\GerbesHom}{\underline{\mathrm{GerbesHom}}}

\newcommand{\spec}{\mathrm{Spec}\,}
\newcommand{\Sp}{\mathrm{Sp}\,}

\newcommand{\CH}{\mathrm{CH}\,}

\newcommand{\br}{\mathrm{Br}\,}
\newcommand{\nr}{\mathrm{nr}}

\newcommand{\brnr}{{\mathrm{Br}_{\nr}}}

\newcommand{\inv}{\mathrm{inv}}

\newcommand{\gm}{\mathbb{G}_{\mathrm{m}}}
\renewcommand{\sl}{\mathrm{SL}}
\newcommand{\gl}{\mathrm{GL}}
\newcommand{\sln}{\mathrm{SL}_n}

\newcommand{\tor}{\mathrm{tor}}
\renewcommand{\ss}{\mathrm{ss}}
\newcommand{\ssu}{\mathrm{ssu}}
\newcommand{\red}{\mathrm{red}}
\newcommand{\torf}{{\rm torf}}
\newcommand{\f}{\mathrm{f}}
\renewcommand{\u}{\mathrm{u}}

\usepackage{marvosym}

\title{Le principe de Hasse pour les espaces homog\`enes : r\'eduction au cas des stabilisateurs finis}
\author{Cyril \textsc{Demarche} et Giancarlo \textsc{Lucchini Arteche}}

\date{}

\begin{document}

\maketitle

\begin{center}
{\it \`A la m\'emoire de Jean-Claude Douai}
\end{center}

\begin{abstract}
Nous montrons, pour une grande famille de propri\'et\'es $P$ des espaces homog\`enes, que $P$ vaut pour tout espace homog\`ene d'un groupe lin\'eaire connexe d\`es qu'elle vaut pour les espaces homog\`enes de $\sln$ \`a stabilisateur fini. Nous r\'eduisons notamment \`a ce cas particulier la v\'erification d'une importante conjecture de Colliot-Th\'el\`ene sur l'obstruction de Brauer-Manin au principe de Hasse et \`a l'approximation faible. Des travaux r\'ecents de Harpaz et Wittenberg montrent que le r\'esultat principal s'applique \'egalement \`a la conjecture analogue (dite conjecture (E)) pour les z\'ero-cycles.

\textbf{Mots cl\'es :} espaces homog\`enes, gerbes, principe de Hasse, obstruction de Brauer-Manin.

\textbf{Classification (MSC 2010) : }  11E72, 14L30.
\end{abstract}

\section{Introduction}\label{section intro}
Soit $k$ un corps de nombres et consid\'erons une famille de $k$-vari\'et\'es lisses et g\'eom\'etriquement int\`egres. On dit qu'une telle famille v\'erifie le principe de Hasse, ou principe local-global, si pour toute vari\'et\'e $X$ dans cette famille on a l'implication suivante :
\begin{equation}\label{PH}
X(k_v)\neq\emptyset,\,\forall\,v\in\Omega_k\,\Rightarrow\, X(k)\neq\emptyset,\tag{PH}
\end{equation}
o\`u $\Omega_k$ d\'esigne l'ensemble des places de $k$ et $k_v$ d\'esigne le compl\'et\'e correspondant \`a la place $v$. La question de la v\'erification du principe local-global est d\'ej\`a classique en g\'eom\'etrie arithm\'etique et remonte aux r\'esultats de Hasse lui-m\^eme sur les quadriques au d\'ebut du XX\textsuperscript{e} si\`ecle.

Une question analogue \`a celle du principe de Hasse est celle de l'approximation faible, laquelle s'int\'eresse \`a la densit\'e des $k$-points d'une vari\'et\'e $X$ dans le produit des points dans ses compl\'et\'es, i.e.
\begin{equation}\label{AF}
X(k)\neq\emptyset\, \Rightarrow\, \overline{X(k)}=\prod_{v\in\Omega_k}X(k_v).\tag{AF}
\end{equation}

Soit maintenant $G$ un $k$-groupe alg\'ebrique lin\'eaire connexe. On rappelle qu'un ($k$-)espace homog\`ene de $G$ est une $k$-vari\'et\'e $X$ munie d'une $k$-action de $G$ qui est transitive au niveau des $\bar k$-points, o\`u $\bar k$ d\'esigne une cl\^oture alg\'ebrique de $k$. Dans ce texte, on s'int\'eresse aux propri\'et\'es arithm\'etiques d\'ecrites ci-dessus pour de telles vari\'et\'es.\\

Les questions sur la validit\'e du principe de Hasse et de l'approximation faible pour les espaces homog\`enes des groupes lin\'eaires ont \'et\'e \'etudi\'ees notamment par Sansuc et Borovoi, avec par exemple la v\'erification des deux propi\'et\'es pour les espaces homog\`enes de groupes semi-simples et simplement connexes \`a stabilisateur semi-simple ou unipotent (cf.~\cite{Borovoi93}). On sait cependant que ce principe n'est pas v\'erifi\'e en g\'en\'eral et c'est d\'ej\`a le cas pour les espaces principaux homog\`enes de tores, ou de groupes semi-simples \`a partir du moment o\`u l'on ne les suppose plus simplement connexes (cf.~\cite{Sansuc81}). En revanche, on peut expliquer ces contre-exemples avec l'obstruction de Brauer-Manin. Dans ce sens, et g\'en\'eralisant les travaux de Sansuc, Borovoi a \'etabli que l'obstruction de Brauer-Manin est la seule obstruction au principe de Hasse et \`a l'approximation faible pour tout espace homog\`ene $X$ de $G$ \`a stabilisateur connexe, ou encore \`a stabilisateur ab\'elien si l'on suppose $G$ simplement connexe (cf.~\cite{Borovoi96}). Autrement dit, on a les propri\'et\'es
\begin{equation}\label{BMPH}
\left[\prod_{v\in\Omega_k} X(k_v)\right]^\brnr\neq\emptyset\Rightarrow X(k)\neq\emptyset,\tag{BMPH}
\end{equation}
et
\begin{equation}\label{BMAF}
X(k)\neq\emptyset\, \Rightarrow\, \overline{X(k)}=\left[\prod_{v\in\Omega_k}X(k_v)\right]^\brnr .\tag{BMAF}
\end{equation}
Pour la d\'efinition de l'\emph{ensemble de Brauer-Manin} $[\prod_{v\in\Omega_k} X(k_v)]^\brnr\subset\prod_{v\in\Omega_k} X(k_v)$, voir la section \ref{section BM}.

La question pour les stabilisateurs non connexes reste cependant ouverte, notamment pour les stabilisateurs finis : \`a propos du principe de Hasse, le cas ab\'elien a \'et\'e trait\'e par Borovoi, les auteurs ont obtenus quelques r\'esultats particuliers (cf.~\cite{Demarche}, \cite{GLA-AFPHEH}), et r\'ecemment Harpaz et Wittenberg ont obtenus une avanc\'ee substantielle dans le cas des stabilisateurs hyper-r\'esolubles (voir \cite{HW}, th\'eor\`eme B). Il est important de remarquer cependant que, d'apr\`es une conjecture de Colliot-Th\'el\`ene sur les vari\'et\'es rationnellement connexes, le m\^eme r\'esultat est attendu dans le cas g\'en\'eral.

Des \'enonc\'es similaires, ouverts eux aussi, peuvent \^etre donn\'es pour les z\'ero-cycles de degr\'e 1. On peut notamment se demander si l'existence d'un z\'ero-cycle de degr\'e 1 sur chaque compl\'et\'e $k_v$ implique l'existence d'un z\'ero-cycle de degr\'e 1 sur $k$, modulo une obstruction de Brauer-Manin. De fa\c con plus pr\'ecise, on conjecture l'exactitude de la suite suivante :
\begin{equation}\label{E}
\varprojlim_n \CH_0(X^c)/n \to \varprojlim_n \prod_{v\in\Omega_k} \CH_0(X^c\times_k k_v)/n \to \hom(\br X^c,\q/\z),\tag{E}
\end{equation}
o\`u $X^c$ est une compactification lisse de $X$ et $\CH_0$ d\'esigne le groupe des z\'ero-cycles modulo \'equivalence rationnelle (si $v$ est une place infinie, on remplace en fait $\CH_0$ par une version modifi\'ee) et la fl\`eche \`a droite est induite par l'accouplement de Brauer-Manin. Ces \'enonc\'es sont connus sous le nom de ``Conjecture (E)'', cf.~par exemple \cite[\S1.1]{Wittenberg}.\\

Le but de cet article est de montrer que l'on peut r\'eduire ces questions g\'en\'erales au cas des espaces homog\`enes de $\sl_{n,k}$ \`a stabilisateur fini. Il s'agit donc de suivre la d\'emarche entreprise par le deuxi\`eme auteur dans \cite{GLA-BMWA}, qui traitait de la propri\'et\'e \eqref{BMAF}. Notons que le passage de \eqref{BMAF} \`a \eqref{BMPH} ou \eqref{E} est d\'elicat car on perd justement l'hypoth\`ese sur l'existence de points rationnels. Il est important de noter que le fait d'enlever l'hypoth\`ese d'existence d'un point rationnel est r\'eellement n\'ecessaire pour obtenir les nouvelles applications. On obtient ainsi par exemple la cons\'equence suivante du r\'esultat principal de ce texte :

\begin{thm}[Cons\'equence du th\'eor\`eme \ref{main thm}]\label{thm principal}
Si l'obstruction de Brauer-Manin est la seule obstruction au principe de Hasse \eqref{BMPH} et \`a l'approximation faible \eqref{BMAF} pour tout espace homog\`ene de $\sl_{n,k}$ \`a stabilisateur fini, alors il en va de m\^eme pour tout espace homog\`ene $X$ d'un groupe lin\'eaire connexe.
\end{thm}

Notons qu'un travail r\'ecent \cite{HW} de Harpaz et Wittenberg d\'emontre la conjecture (E) pour les espaces homog\`enes de groupes lin\'eaires connexes en utilisant le th\'eor\`eme \ref{main thm} pour montrer qu'essentiellement l'exactitude de la suite \eqref{E} pour les espaces homog\`enes se r\'eduit au cas des espaces homog\`enes de $\sl_{n,k}$ \`a stabilisateurs finis.

On remarquera que, \`a la diff\'erence du r\'esultat principal de \cite{GLA-BMWA}, qui correspond au m\^eme \'enonc\'e que le th\'eor\`eme \ref{main thm}, mais seulement avec la propri\'et\'e \eqref{BMAF}, on n'a pas un \'enonc\'e analogue avec seulement la propri\'et\'e \eqref{BMPH}. En effet, on a besoin d'un minimum de propri\'et\'es d'approximation pour que notre preuve, qui utilise la m\'ethode des fibrations, fonctionne. Le mieux que l'on puisse faire pour l'instant est de remplacer \eqref{BMAF} par ``approximation r\'eelle'' dans l'\'enonc\'e ci-dessus, ce qui est une hypoth\`ese plus faible, mais non triviale tout de m\^eme, cf.~le d\'ebut de la section \ref{section thm principal}.

La preuve de ce r\'esultat s'inspire de celle dans \cite{GLA-BMWA}, mais il est n\'ecessaire de g\'en\'eraliser plusieurs outils pr\'esents dans celle-ci. Le point crucial est que dans un espace homog\`ene ayant un point rationnel, le stabilisateur de ce point est un $k$-groupe, alors qu'en g\'en\'eral, on dispose seulement d'une $k$-gerbe. La nouveaut\'e principale de ce texte est donc de donner un sens \`a des phrases telles que ``plonger une $k$-gerbe dans $\sl_{n,k}$'' ou ``faire agir une $k$-gerbe sur un tore''.\\

Le plan du texte est le suivant. Dans la section \ref{section prel}, on donne un rappel succinct sur l'obstruction de Brauer-Manin et un rappel plus complet sur la 2-cohomologie non ab\'elienne. On d\'efinit notamment, pour $K$ un corps parfait, les notions de $K$-lien et $K$-gerbe de plusieurs fa\c cons d\'ej\`a classiques mais qui, \`a notre connaissance, n'ont jamais \'et\'e toutes compar\'ees en un seul texte dans la litt\'erature. On conclut cette section en rappelant le lien entre la 2-cohomologie et les espaces homog\`enes sous ces diff\'erents points de vue. La section \ref{section gerbes dans SLn} est consacr\'ee au ``plongement d'une $K$-gerbe dans $\sl_{n,K}$''. On d\'emontre au passage une g\'en\'eralisation du fameux ``Lemme sans nom'' (Cor.~\ref{corollaire lemme sans nom}), qui affirme que deux espaces homog\`enes de $\sl_{n,K}$ ou de $\gl_{n,K}$ ayant la m\^eme $K$-gerbe associ\'ee sont $K$-stablement birationnels. La section \ref{section ono} s'occupe de l'action d'une $K$-gerbe sur un tore. On reprend ici ce qui a \'et\'e fait dans \cite[\S3]{GLA-BMWA} autour du lemme d'Ono pour l'\'etendre \`a ce nouveau cadre. Enfin, la section \ref{section thm principal} est consacr\'ee \`a l'\'enonc\'e et la preuve du th\'eor\`eme principal.

\paragraph*{Remerciements}
Les auteurs tiennent \`a remercier Olivier Wittenberg et Jean-Louis Colliot-Th\'el\`ene pour leurs commentaires.

Le premier auteur a b\'en\'efici\'e d'une aide de l'Agence Nationale de la Recherche portant les r\'ef\'erences ANR-12-BL01-0005 et ANR-15-CE40-0002. Le travail du deuxi\`eme auteur a \'et\'e partiellement soutenu par la FMJH via la bourse No ANR-10-CAMP-0151-02 dans le Programme des Investissements d'Avenir.

\section{Pr\'eliminaires}\label{section prel}

Dans tout le texte, on d\'esigne par $K$ un corps parfait, $\bar K$ une cl\^oture alg\'ebrique et $\Gamma_K$ le groupe de Galois absolu $\gal(\bar K/K)$. On r\'eserve la notation $k$ pour un corps de nombres. On note alors $\Omega_k$ l'ensemble des places de $k$ et, pour $v\in\Omega_k$, on note $k_v$ le compl\'et\'e correspondant.

Une $K$-vari\'et\'e alg\'ebrique $X$ sera toujours suppos\'ee lisse et g\'eom\'etriquement int\`egre. Pour $A/K$ une $K$-alg\`ebre, on note $X_A$ la $A$-vari\'et\'e obtenue par changement de base. Lorsque $A=\bar K$, on utilise plut\^ot $\bar X$ au lieu de $X_{\bar K}$ et on garde plus g\'en\'eralement cette notation pour toute vari\'et\'e d\'efinie sur $\bar K$, m\^eme si elle n'admet pas forc\'ement de $K$-forme. Il en va de m\^eme pour les $K$-groupes alg\'ebriques, les $K$-liens et les $K$-gerbes, not\'es respectivement $G$, $L$ et $\cal M$ dans la suite.

Si $G$ est un $K$-groupe alg\'ebrique lin\'eaire, on note
\begin{itemize}
\item $G^\circ$ la composante neutre de $G$;
\item $G^\f=G/G^\circ$ le groupe (fini) des composantes connexes de $G$;
\item $G^\u$ le radical unipotent de $G^\circ$;
\item $G^\red=G^\circ/G^\u$, qui est un groupe r\'eductif;
\item $G^\ss=D(G^\red)=[G^\red, G^\red]$, qui est un groupe semi-simple;
\item $G^\tor=G^\red/G^\ss$, qui est un tore;
\item $G^\ssu=\ker[G^\circ\twoheadrightarrow G^\tor]$, qui est une extension de $G^\ss$ par $G^\u$;
\item $G^\torf=G/G^\ssu$, qui est une extension de $G^\f$ par $G^\tor$;
\end{itemize}
qui sont tous d\'efinis sur $K$.

\subsection{L'obstruction de Brauer-Manin}\label{section BM}
Pour des d\'etails sur l'obstruction de Brauer-Manin, on renvoie le lecteur vers \cite[\S 5.1]{Skor}. Pour une vari\'et\'e $X$ sur un corps de nombres $k$, on consid\`ere le groupe de Brauer non ramifi\'e $\brnr X$. Ce dernier correspond au groupe de Brauer d'une compactification lisse de $X$ (laquelle existe toujours sur un corps de caract\'eristique nulle d'apr\`es les th\'eor\`emes de Nagata et Hironaka). Pour $v\in\Omega_k$, on d\'esigne par $\inv_v:\br k_v \to \q/\z$ l'invariant de Hasse tel qu'il est d\'efini par la th\'eorie du corps de classes locale. On consid\`ere le plongement diagonal de $X(k)$ dans le produit $\prod_{v\in\Omega_k} X(k_v)$. Alors on peut d\'efinir le sous-ensemble $[\prod_{v\in\Omega_k} X(k_v)]^{\brnr}$ des familles de points locaux telles que
\[\sum_{v\in\Omega_k}\inv_v(\alpha(P_v))=0,\quad\forall\,\alpha\in\brnr X,\]
o\`u $\alpha(P_v)\in\br k_v$ d\'esigne l'application d'\'evaluation de $\alpha$ au point $P_v$. Remarquons que la somme pr\'ec\'edente est finie pour les \'el\'ements $\alpha\in\brnr X$. L'ensemble $[\prod_{v\in\Omega_k} X(k_v)]^{\brnr}$ est dit \emph{ensemble de Brauer-Manin}. La th\'eorie du corps de classes globale nous fournit alors les inclusions suivantes :
\[\overline{X(k)}\subseteq \left[\prod_{v\in\Omega_k} X(k_v)\right]^{\brnr}\subseteq \prod_{v\in\Omega_k} X(k_v),\]
o\`u $\overline{X(k)}$ d\'esigne l'adh\'erence de $X(k)$ dans le produit. Puisque l'approximation faible \'etablit l'\'egalit\'e entre les ensembles de droite et de gauche, l'ensemble de Brauer-Manin fournit une obstruction \`a l'approximation faible d\`es que l'inclusion de droite est stricte. De m\^eme, si l'ensemble de Brauer-Manin est vide alors que le produit tout entier ne l'est pas, on trouve une obstruction au principe de Hasse.

Une conjecture de Colliot-Th\'el\`ene implique l'\'egalit\'e $\overline{X(k)}=[\prod_{v\in\Omega_k} X(k_v)]^{\brnr}$ pour tout espace homog\`ene $X$ d'un groupe lin\'eaire connexe $G$ (cf.~\cite[Intro]{ColliotBudapest}). On dit alors que l'obstruction de Brauer-Manin est la seule obstruction au principe de Hasse et \`a l'approximation faible.

\subsection{2-cohomologie non ab\'elienne}\label{section H2 nonab}

Soit $K$ un corps parfait. On fixe une cl\^oture alg\'ebrique $\bar K$ de $K$ et on note $\Gamma_K := \gal(\bar K / K)$. Dans toute la suite, on munit $\spec(K)$ du petit site \'etale.

Les notions de $K$-lien et de $K$-gerbe sont importantes pour d\'efinir et \'etudier la cohomologie galoisienne en degr\'e $2$ des groupes alg\'ebriques non commutatifs. On pr\'esente ici plusieurs points de vue diff\'erents, et n\'eanmoins \'equivalents, sur ces notions, que l'on utilisera dans la suite du texte.

On rappelle que tous les groupes alg\'ebriques sont suppos\'es lisses.

\subsubsection{Liens}
On dispose de deux notions de $K$-lien : la notion abstraite de Giraud (cf.~\cite[IV.1]{Giraud}) et la version concr\`ete de Borovoi (cf.~\cite[\S1]{Borovoi93}), dont une premi\`ere version avait \'et\'e pr\'esent\'ee par Springer dans \cite{SpringerH2} et dont la g\'en\'eralisation aux corps de caract\'eristique positive se trouve dans \cite{FSS}. 

\paragraph{Liens au sens de Giraud.}
On consid\`ere la cat\'egorie fibr\'ee (et m\^eme scind\'ee) sur le petit site \'etale de $\spec K$, not\'ee $\underline{\mathrm{LI}}(K)$, d\'efinie de la fa\c con suivante : la cat\'egorie fibre en $\spec A$ pour une $K$-alg\`ebre \'etale $A$ a pour objets les faisceaux en groupes sur $\spec A$, et si $H$ et $G$ sont des $A$-faisceaux en groupes, un morphisme de $H$ vers $G$ dans $\underline{\mathrm{LI}}(K)(A)$ est une section sur $\spec A$ du faisceau quotient $\underline{\mathrm{Int}}(G) \backslash \underline{\mathrm{Hom}}_{A\textup{-gr}}(H,G)$. On d\'efinit alors le champ (scind\'e) des $K$-liens, not\'e $\underline{\mathrm{LIENS}}(K)$, comme le champ associ\'e au pr\'echamp $\underline{\mathrm{LI}}(K)$. Un lien sur $\spec K$ est alors un objet de la cat\'egorie fibre $\underline{\mathrm{Liens}}(K)$ de $\underline{\mathrm{LIENS}}(K)$ en $\spec K$. On dispose en particulier d'un morphisme de $K$-champs (scind\'es)
\[\underline{\mathrm{lien}} : \underline{\mathrm{FAISCGR}}(K) \to \underline{\mathrm{LIENS}}(K),\]
qui \`a un faisceau en groupes $G$ associe le lien, not\'e $\underline{\mathrm{lien}}(G)$, d\'efini par ce faisceau.

Pour la suite, on note $\underline{\mathrm{Gralg}}(\bar K)$ la cat\'egorie dont les objets sont les $\bar K$-groupes alg\'ebriques lin\'eaires lisses, et les morphismes d'un $\bar K$-groupe $\bar H$ vers un $\bar K$-groupe $\bar G$ sont les morphismes de $\bar K$-groupes alg\'ebriques $\bar H \to \bar G$ modulo $\mathrm{Int}(\bar G)$. On dispose alors d'un foncteur naturel $\underline{\textup{lien}} : \underline{\mathrm{Gralg}}(\bar K) \to \underline{\mathrm{Liens}}(\bar K)$ r\'ealisant  la premi\`ere cat\'egorie comme une sous-cat\'egorie de la seconde.

On peut alors d\'efinir la cat\'egorie $\underline{\mathrm{Lialg}}(K)$ comme le produit 2-fibr\'e suivant  (voir par exemple \cite[\href{https://stacks.math.columbia.edu/tag/02X9}{02X9}]{SP}) : 
$$\underline{\mathrm{Lialg}}(K) := \underline{\mathrm{Liens}}(K) \times_{\underline{\mathrm{Liens}}(\bar K)} \underline{\mathrm{Gralg}}(\bar K) \, .$$
Plus concr\`etement, un objet de $\underline{\mathrm{Lialg}}(K)$ est un triplet $(L, \bar G, \varphi)$, o\`u $L$ est un lien sur $\spec(K)$, $\bar G$ un $\bar K$-groupe lin\'eaire lisse et $f$ un isomorphisme $\varphi : L_{\bar K} \xrightarrow{\sim} \underline{\textup{lien}}(\bar G)$ de liens sur $\spec(\bar K)$; un morphisme $(M, \bar H, \psi) \to (L, \bar G, \varphi)$ est un couple $(f,g)$, o\`u $f : M \to L$ est un morphisme de liens sur $\spec(K)$ et $g : \bar H \to \bar G$ un morphisme de $\bar K$-groupes alg\'ebriques modulo $\mathrm{Int}(\bar H)$, tel que $\varphi \circ f_{\bar K} = \underline{\textup{lien}}(g) \circ \psi$.
Les objets $(L, \bar G, \varphi)$ de $\underline{\mathrm{Lialg}}(K)$ sont appel\'es \emph{liens alg\'ebriques} sur $\spec K$, et ils seront parfois not\'es abusivement $L$.\\

Remarquons enfin que, puisque toute $\bar K$-alg\`ebre \'etale est un produit direct de copies de $\bar K$, le corollaire~9.29 de \cite{GiraudDesc} assure qu'une donn\'ee de descente sur une $K$-alg\`ebre \'etale $A=\prod_{i\in I}K_i$ avec $K_i/K$ des extensions de corps correspond \`a des donn\'ees de descente respectives sur chacun des $K_i$. On en d\'eduit que pour \'etudier les champs sur le petit site \'etale de $\spec K$, il suffit de se restreindre aux extensions finies de $K$ (qui sont toutes s\'eparables puisque $K$ est parfait).

\paragraph{Liens au sens de Springer-Borovoi.}
Pr\'esentons maintenant la seconde notion de $K$-lien. Soit $\bar G$ un $\bar K$-groupe alg\'ebrique lin\'eaire lisse. Rappelons (cf.~\cite[\S1]{Borovoi93} ou encore \cite[\S1]{FSS}) que le groupe $\saut(\bar G / K)$ des automorphismes semi-alg\'ebriques de $\bar G$ sur $K$ (not\'e aussi $\saut(\bar G)$ lorsque $K$ est sous-entendu) est form\'e des diagrammes commutatifs de la forme
\[\xymatrix{
\bar G \ar[r]^s \ar[d] & \bar G \ar[d] \\
\spec \bar K \ar[r]^{\Sp\sigma^{-1}} & \spec \bar K,
}\]
pour un certain $\sigma\in\Gamma_K$ et o\`u $\Sp\sigma$ d\'esigne l'automorphisme de $\spec \bar K$ induit par $\sigma$ sur le corps $\bar K$. On peut aussi voir ces objets comme des morphismes de $\bar K$-groupes $\sigma_*\bar G\to \bar G$, o\`u $\sigma_*\bar G$ est le $\bar K$-groupe alg\'ebrique $\bar G$ muni du morphisme structurel ``tordu''
\[\bar G\to\spec\bar K\xrightarrow{\Sp\sigma^{-1}}\spec \bar K.\]
Ou encore, en consid\'erant le $\bar K$-groupe $(\Sp\sigma)^*\bar G$, on peut v\'erifier ais\'ement qu'il est \emph{canoniquement} isomorphe \`a $\sigma_*\bar G$ en tant que $\bar K$-groupe, ce qui fait que l'on peut remplacer $\sigma_*\bar G$ par $(\Sp\sigma)^*\bar G$ ci-dessus.

Enfin, on rappelle que l'on dispose d'une suite exacte
\[1\to\aut_{\bar K}(\bar G)\to\saut(\bar G)\to\Gamma_K,\]
ainsi que du groupe quotient $\sout(\bar G):=\saut(\bar G)/\int(\bar G)$, o\`u $\int(\bar G)$ d\'esigne l'image du morphisme $\bar G\to\aut(\bar G)$ induit par la conjugaison.

Un $K$-lien sur $\bar G$ est la donn\'ee d'un morphisme de groupes continu $\kappa : \Gamma_K \to \sout(\bar G)$ qui scinde la suite exacte naturelle
\[1\to \out_{\bar K}(\bar G)\to \sout(\bar G) \to \Gamma_K,\]
et qui se rel\`eve en une section continue $f:\Gamma_K \to \saut(\bar G)$, On rappelle la d\'efinition de continuit\'e dans ce contexte (cf.~\cite[1.10]{FSS}). On sait que $\bar G$ admet une $K'$-forme pour une certaine extension finie $K'/K$. Cette $K'$-forme d\'efinit naturellement (cf.~\cite[1.4]{FSS}) un scindage $s:\Gamma_{K'}\to\sout(\bar G/K')\subset\sout(\bar G/K)$. L'application $f$ est dite continue si le morphisme
\[\Gamma_{K'}\to\aut_{\bar K}(\bar G):s_\tau^{-1}f_\sigma^{-1}f_{\sigma\tau},\]
est localement constant pour tout $\sigma\in\Gamma_K$.

On dispose \'egalement dans ce contexte de la notion suivante de morphisme de $K$-liens : si $(\bar G, \kappa)$ et $(\bar H, \lambda)$ sont deux $K$-liens, consid\'erons les morphismes de $\bar K$-groupes $\varphi : \bar G \to \bar H$ tels qu'il existe des relev\'es continus $g : \Gamma_K \to \saut(\bar G)$ et $h : \Gamma_K \to \saut(\bar H)$ de $\kappa$ et $\lambda$ respectivement, de sorte que pour tout $\sigma \in \Gamma_K$, le diagramme suivant de $\bar K$-groupes
\[
\xymatrix{
{\sigma_*\bar G} \ar[r]^{g_\sigma} \ar[d]_{\sigma_*\varphi} & {\bar G} \ar[d]^\varphi \\
{\sigma_*\bar H} \ar[r]^{h_\sigma} & {\bar H},
}
\]
commute (et o\`u la signification de $\sigma_*\varphi$ est \'evidente). On a alors une action naturelle de $\int(\bar H)$ sur ces morphismes, ce qui permet de d\'efinir un morphisme de liens $\varphi : (\bar G, \mu) \to (\bar H, \kappa)$ comme une classe de morphismes au sens pr\'ec\'edent modulo $\int(\bar H)$. On a ainsi d\'efini une cat\'egorie, not\'ee $\underline{K\textup{-Liens}}$.

\paragraph{Comparaison I.}
On dispose d'une \'equivalence de cat\'egories naturelle $\underline{\mathrm{Lialg}}(K) \to \underline{K\textup{-Liens}}$. Elle est d\'efinie ainsi : soit $(L,\bar G, \varphi)$ un objet de $\underline{\mathrm{Lialg}}(K)$. Puisque $L$ est un lien sur $\spec(K)$ et $\varphi : L_{\bar K} \to \underline{\textup{lien}}(\bar G)$ un isomorphisme de liens sur $\bar K$, on dispose d'une donn\'ee de descente naturelle sur $\underline{\textup{lien}}(\bar G)$ pour l'extension $\bar K / K$, c'est-\`a dire que pour tout $\sigma \in \Gamma_K$, on a un isomorphisme canonique dans $\underline{\mathrm{Liens}}(\bar K)$ de la forme
\[f_{\sigma} : \underline{\mathrm{lien}}(\sigma_* \bar G) \xrightarrow{(\Sp\sigma)^* \varphi^{-1}} (\Sp\sigma)^* L_{\bar K}  \to L_{\bar K} \xrightarrow{\varphi} \underline{\mathrm{lien}}(\bar G),\]
v\'erifiant $f_{\sigma \tau} = f_{\sigma} \circ {(\Sp\sigma)^* f_{\tau}}$, o\`u la premi\`ere fl\`eche se d\'eduit de \cite[V.1.2.2.2]{Giraud} et de l'isomorphisme $(\Sp\sigma)^*\bar G\cong\sigma_*\bar G$. Or, puisque
\[\mathrm{Isom}_{\underline{\mathrm{Liens}}(\bar K)}(\underline{\mathrm{lien}}(\sigma_* \bar G), \underline{\mathrm{lien}}(\bar G)) = \int(\bar G) \backslash \mathrm{Isom}_{\bar K\textup{-gr}}(\sigma_* \bar G, \bar G),\]
et que $\mathrm{Isom}_{\bar K\textup{-gr}}(\sigma_* \bar G, \bar G)$ correspond \`a la pr\'eimage de $\sigma\in\Gamma_K$ par rapport au morphisme naturel $\saut(\bar G)\to\Gamma_K$, on en d\'eduit que la donn\'ee de descente $(f_\sigma)$ d\'efinit un morphisme de groupes $\kappa_L : \Gamma_K \to \sout(\bar G)$ scindant le morphisme $\sout(\bar G) \to \Gamma_K$, dont on v\'erifie qu'il est continu et qu'il se rel\`eve en une application continue $\Gamma_K \to \saut(\bar G)$ (on utilise que l'isomorphisme entre $L_{\bar K}$ et $\underline{\mathrm{lien}}(\bar G)$ est d\'efini localement pour la topologie \'etale, donc sur une extension finie de $K$). Donc $(\bar G, \kappa_L)$ est un $K$-lien.

V\'erifions que l'on a bien d\'efini un foncteur $\underline{\mathrm{Lialg}}(K) \to \underline{K\textup{-Liens}}$. Soit un morphisme $\tau : L=(L,\bar G, \varphi) \to M=(M,\bar H,\psi)$ dans $\underline{\mathrm{Lialg}}(K)$. On a construit les $K$-liens $(\bar G, \kappa_L)$ et $(\bar H, \kappa_M)$ associ\'es respectivement \`a $L$ et $M$. La construction assure que $\tau$ induit un morphisme de $\bar K$-groupes alg\'ebriques $\bar \tau : \bar G \to \bar H$ d\'efini modulo $\int(\bar H)$. 
On a un diagramme commutatif dans $\underline{\mathrm{Liens}}(\bar K)$ :
\[
\xymatrix{
L_{\bar K} \ar[r]^{\tau_{\bar K}} \ar[d]^{\varphi} & M_{\bar K} \ar[d]^{\psi} \\
\underline{\mathrm{lien}}(\bar G) \ar[r]^{\bar \tau} &  \underline{\mathrm{lien}}(\bar H) \, . 
}
\]
Or $\tau : L \to M$ est un morphisme dans $\underline{\mathrm{Liens}}(K)$, donc $\bar \tau$ est muni d'une donn\'ee de descente naturelle pour l'extension $\bar K / K$, c'est-\`a dire que pour tout $\sigma \in \Gamma_K$, on a un diagramme commutatif dans $\underline{\mathrm{Liens}}(\bar K)$ de la forme
\[\xymatrix{
(\Sp\sigma)^*\underline{\mathrm{lien}}(\bar G) \ar[d]^{(\Sp\sigma)^*{\bar \tau}} \ar@{=}[r] & \underline{\mathrm{lien}}(\sigma_* \bar G) \ar[d]^{\sigma_*{\bar \tau}} \ar[r]^{g_\sigma} & \underline{\mathrm{lien}}(\bar G) \ar[d]^{\bar \tau} \\
(\Sp\sigma)^*\underline{\mathrm{lien}}(\bar H) \ar@{=}[r] & \underline{\mathrm{lien}}(\sigma_* \bar H) \ar[r]^{h_\sigma} & \underline{\mathrm{lien}}(\bar H),
}\]
o\`u $g_\sigma$ et $h_\sigma$ proviennent (via $\varphi$ et $\psi$) des donn\'ees de descente respectives pour $L_{\bar K}$ et $M_{\bar K}$. On v\'erifie alors que ces donn\'ees de descente d\'efinissent un morphisme $(\bar G_L, \kappa_L) \to (\bar G_M, \kappa_M)$ dans la cat\'egorie $\underline{K\textup{-Liens}}$. D'o\`u un foncteur $\underline{\mathrm{Lialg}}(K) \to \underline{K\textup{-Liens}}$.

Montrons que le foncteur ainsi d\'efini est une \'equivalence de cat\'egories. Un quasi-inverse est donn\'e par le foncteur suivant : si $(\bar G, \kappa)$ est un $K$-lien, alors il existe une extension finie galoisienne $K'/K$ et une $K'$-forme $G'$ de $\bar G$ correspondant \`a un scindage continu
\[\Gamma_{K'}\to\saut(\bar G/K')\subset\saut(\bar G/K),\]
du morphisme $\saut(\bar G/K')\to\Gamma_{K'}$, cf.~\cite[1.4]{FSS}. Par continuit\'e (cf.~\cite[1.10]{FSS}) on sait que, quitte \`a agrandir $K'$, on peut supposer que ce scindage rel\`eve $\kappa$. On dispose alors de l'objet $\underline{\mathrm{lien}}(G')$ dans la cat\'egorie $\underline{\mathrm{Liens}}(K')$ qui est bien d\'efini \`a isomorphisme \emph{unique} pr\`es. En effet, par continuit\'e, deux scindages diff\'erents induisant des $K'$-formes $G'_1$ et $G'_2$ sont conjugu\'es sur une extension galoisienne finie $K''/K'$. Cette donn\'ee, jointe aux isomorphismes $\underline{\mathrm{lien}}(\bar G_i')\to\underline{\mathrm{lien}}(\bar G)$ sur $\bar K$ fixent un unique isomorphisme de liens sur $K''$. Le fait que les deux scindages rel\`event $\kappa$ induit une donn\'ee de descente sur ce $K''$-morphisme, induisant un unique $K'$-isomorphisme $\underline{\mathrm{lien}}(G'_1)\to\underline{\mathrm{lien}}(G'_2)$. Ensuite, on peut d\'efinir un rel\`evement continu de l'application $\kappa:\Gamma_K \to \sout(\bar G)$ en prenant le scindage sur $\Gamma_{K'}$ et des translat\'es de ce scindage pour les classes lat\'erales de $\Gamma_{K'}\lhd\Gamma_K$. On en d\'eduit ainsi pour chaque $\sigma\in\Gamma_{K'/K}$ un morphisme
\[f_\sigma : (\Sp\sigma)^* \underline{\mathrm{lien}}(G') = \underline{\mathrm{lien}}(\sigma_* \bar G') \to \underline{\mathrm{lien}}(G'),\]
dans la cat\'egorie $\mathrm{Liens}(K')$. Puisque $\kappa$ est un morphisme de groupes, ces morphismes v\'erifient $f_{\sigma \tau} = f_\sigma \circ ((\Sp\sigma)^* f_\tau)$. Par cons\'equent, les morphismes $(f_\sigma)$ d\'efinissent une donn\'ee de descente sur $\underline{\mathrm{lien}}(G')$. Comme la cat\'egorie fibr\'ee $\underline{\mathrm{LIENS}}(K)$ des liens est un $K$-champ, cette donn\'ee de descente d\'efinit un objet $L_\kappa$ de $\underline{\mathrm{Liens}}(K)$ et un isomorphisme $\varphi : L_{\kappa,K'} \to \underline{\textup{lien}}(G')$, lequel induit un isomorphisme $\varphi : L_{\kappa,\bar K} \to \underline{\textup{lien}}(\bar G)$ et donc $(L_\kappa,\bar G, \varphi)$ est dans $\underline{\textup{Lialg}}(K)$. De m\^eme, si $\tau : (\bar G, \kappa) \to (\bar H, \lambda)$ est un morphisme dans $\underline{K\textup{-Liens}}$, on v\'erifie que celui-ci d\'efinit un \emph{unique} morphisme $\tau':\underline{\mathrm{lien}}(G')\to\underline{\mathrm{lien}}(H')$ pour des $K'$-formes $G'$ et $H'$ de $\bar G$ et $\bar H$ respectivement. On v\'erifie ensuite que ce morphisme dans $\underline{\mathrm{Lialg}}(K')$ est muni d'une donn\'ee de descente, ce qui permet d'en d\'eduire un morphisme $\widetilde{\tau} : L_\kappa \to L_\lambda$ puisque $\underline{\mathrm{LIENS}}(K)$ est un $K$-champ. Cela d\'efinit bien le foncteur souhait\'e $\underline{K\textup{-Liens}} \to \underline{\mathrm{Lialg}}(K)$. On peut alors v\'erifier que celui-ci est bien un quasi-inverse du premier.

\subsubsection{2-cohomologie non ab\'elienne}\label{section 2-coh}
On dispose de trois fa\c cons \'equivalentes de voir la 2-cohomologie galoisienne non ab\'elienne. Les deux premi\`eres, correspondant aux points de vue des cocycles et des extensions, sont bien connues dans le cadre de la cohomologie des groupes non ab\'elienne classique et s'adaptent donc au cadre de la cohomologie galoisienne via la notion de $K$-lien \`a la Springer-Borovoi. Le troisi\`eme point de vue, celui des gerbes, est issu du point de vue plus abstrait des liens \`a la Giraud.

\paragraph{Cocycles.}
Soit $L = (\bar G, \kappa)$ un $K$-lien. On munit $\bar G (\bar K)$ de la topologie discr\`ete. Un $2$-cocycle \`a valeurs dans $L$ est un couple $(f,u)$, o\`u $f : \Gamma_K \to \saut(\bar G)$ est une application continue qui rel\`eve $\kappa$ et $u : \Gamma_K \times \Gamma_K \to \bar G (\bar K)$ est une application continue v\'erifiant, pour tout $\sigma, \tau, \upsilon \in \Gamma_K$,
\begin{gather*}
f_{\sigma, \tau} = \int (u_{\sigma,\tau}) \circ f_{\sigma} \circ f_{\tau},\\
u_{\sigma, \tau \upsilon} \cdot f_{\sigma}(u_{\tau, \upsilon}) = u_{\sigma \tau, \upsilon} \cdot u_{\sigma, \tau},
\end{gather*}
o\`u $\int (u_{\sigma,\tau})$ d\'esigne le morphisme de conjugaison par l'\'el\'ement $u_{\sigma,\tau}$.

Deux $2$-cocycles $(f,u)$ et $(f',u')$ sont dits \'equivalents s'il existe une application continue $c : \Gamma_K \to \bar G (\bar K)$ telle que pour tout $\sigma, \tau \in \Gamma_K$, on a
\begin{gather*}
f'_\sigma = \int(c_{\sigma}) \circ f_{\sigma} \\
u'_{\sigma, \tau} = c_{\sigma \tau} \cdot u_{\sigma, \tau} \cdot f_\sigma(c_\tau)^{-1} \cdot c_\sigma^{-1} \, .
\end{gather*}

L'ensemble des $2$-cocycles modulo cette relation d'\'equivalence est not\'e $H^2(K, L)$, ou $H^2(K,G)$ si $L=\underline{\mathrm{lien}}(G)$. Une classe dans $H^2(K, L)$ est dite neutre si elle est repr\'esentable par un $2$-cocycle de la forme $(f,1)$. Elle correspond alors \`a une $k$-forme $G$ de $\bar G$ (cf.~\cite[1.4]{Borovoi93}) et on la note alors $n(G)$.

\paragraph{Extensions.}
Soit $L = (\bar G, \kappa)$ un $K$-lien. On munit toujours $\bar G (\bar K)$ de la topologie discr\`ete. On d\'efinit une extension li\'ee par $L$ comme une suite exacte de groupes topologiques
\[1 \to \bar G (\bar K) \to E \to \Gamma_K \to 1 \, ,\]
telle que le morphisme naturel $\Gamma_K \to \out(G(\bar K))$ induit par cette suite exacte soit \'egal \`a la compos\'ee de $\kappa$ avec le morphisme naturel $\sout(\bar G) \to \out(G(\bar K))$. On rappelle que par suite exacte de groupes topologiques on entend que la fl\`eche de gauche induit un hom\'eomorphisme avec son image, qui est un ferm\'e de $E$, alors que la fl\`eche de droite est un morphisme ouvert.

On note $\ext(\Gamma_K, L)$ l'ensemble des classes d'\'equivalence de telles extensions. Une extension est dite neutre si elle admet une section qui est un morphisme continu.

\paragraph{Gerbes.}
Une $K$-gerbe est un champ en groupo\"ides $\mathcal{M}$ sur le petit site \'etale de $\spec(K)$ tel que 
\begin{itemize}
\item il existe une $K$-alg\`ebre \'etale $A / K$ telle que $\mathcal{M}(A) \neq \emptyset$.
\item pour toute $K$-alg\`ebre \'etale $A$, pour tout $m, m' \in \mathcal{M}(A)$, il existe une $A$-alg\`ebre \'etale $A'$ et un isomorphisme $m \xrightarrow{\sim} m'$ dans $\mathcal{M}(A')$.
\end{itemize}
Comme on l'a mentionn\'e dans la section pr\'ec\'edente, on peut d\'eduire de \cite[Cor.~9.29]{GiraudDesc} que l'on peut remplacer ``alg\`ebre \'etale'' par ``extension finie'' dans les deux points ci-dessus.

Par exemple, pour tout $K$-groupe alg\'ebrique $G$ (ou plus g\'en\'eralement un faisceau \'etale en groupes), le champ des $G$-torseurs, not\'e $\TORS(G)$, est une $K$-gerbe. 

\`A une $K$-gerbe $\mathcal{M}$ on associe son lien $L_\mathcal{M}$ d\'efini de la fa\c con suivante : on choisit $m \in \mathcal{M}(K')$ pour une extension finie galoisienne $K'$ de $K$. Pour tout $\sigma \in \Gamma_K$, on dispose du foncteur naturel $\mathcal{M}(K') \to \mathcal{M}(K')$ induit par $\Sp\sigma$. On note ${^\sigma m}$ l'image de $m$ par ce foncteur. Par d\'efinition, il existe une extension finie $K''/K'$ telle que pour tout $\sigma \in \gal(K'/K)$, il existe un (iso)morphisme $\varphi_\sigma : {^\sigma m} \to m$ dans la cat\'egorie $\mathcal{M}(K'')$. On note alors $f_\sigma : \underline{\aut}({^\sigma m}) \to \underline{\aut}(m)$ le morphisme de $K''$-faisceaux en groupes d\'efini par $f_\sigma(\alpha) := \varphi_\sigma \circ \alpha \circ {\varphi_\sigma}^{-1}$. Or, d'apr\`es \cite[V.1.3.3.3]{Giraud}, on a un isomorphisme naturel $\underline{\aut}({^\sigma m}) \cong (\Sp\sigma)^* \underline{\aut}(m)$. Donc on a d\'efini un morphisme $f_\sigma : (\Sp\sigma)^* \underline{\aut}(m) \to \underline{\aut}(m)$ de $K''$-faisceau en groupes, donc un morphisme de $K''$-liens $f_\sigma : (\Sp\sigma)^* \underline{\mathrm{lien}}(\underline{\aut}(m)) \to \underline{\mathrm{lien}}(\underline{\aut}(m))$. On v\'erifie que cela munit $\underline{\mathrm{lien}}(\underline{\aut}(m))$ d'une donn\'ee de descente relativement \`a l'extension $K''/ K$, donc cela d\'efinit un lien sur $\spec K$ que l'on note $L_\mathcal{M}$ et, \`a isomorphisme unique pr\`es, ce lien ne d\'epend pas des choix de $K'$, $m$, $K''$, ni du choix des isomorphismes $\varphi_\sigma$. Si $\underline{\aut}(m)$ est repr\'esentable par un $\bar K$-groupe lin\'eaire lisse, alors on dispose d'un lien alg\'ebrique $L(\mathcal{M}) = (L_\mathcal{M}, \underline{\aut}(m), \id_{\underline{\textup{lien}}(\underline{\aut}(m))})$, bien d\'efini \`a isomorphisme unique pr\`es dans $\underline{\textup{Lialg}}(K)$. Dans ce cas, on dit que $\mathcal{M}$ est une gerbe alg\'ebrique. 

Un morphisme de $K$-gerbes est un morphisme de $K$-champs dont la source et le but sont des gerbes. Tout morphisme $\varphi : \mathcal{M} \to \mathcal{N}$ de $K$-gerbes (resp. de $K$-gerbes alg\'ebriques) induit un morphisme de $K$-liens (resp. de $K$-liens alg\'ebriques) $\underline{\mathrm{lien}}(\varphi) : L_\mathcal{M} \to L_\mathcal{N}$ (resp. $\underline{\mathrm{lien}}(\varphi) : L(\mathcal{M}) \to L(\mathcal{N})$), cf.~\cite[IV.2.2.3]{Giraud}. Le morphisme $\varphi$ est dit li\'e par $\underline{\mathrm{lien}}(\varphi)$. Une \'equivalence (resp.~un morphisme injectif, resp.~surjectif) de $K$-gerbes est un morphisme de gerbes li\'e par un isomorphisme (resp. par un morphisme injectif, resp.~surjectif) de liens. Bien entendu, un morphisme de liens est dit injectif (resp.~surjectif) si le morphisme de $\bar K$-groupes alg\'ebriques sous-jacent l'est.

Soit $L=(L, \bar G, \varphi)$ un lien alg\'ebrique sur $\spec K$ au sens de Giraud. Une $K$-gerbe li\'ee par $L$ est un couple $(\mathcal{M}, \rho)$, o\`u $\mathcal{M}$ est une $K$-gerbe et $\rho : L(\mathcal{M}) \to L$ un isomorphisme de liens alg\'ebriques. Deux telles $K$-gerbes $(\mathcal{M}, \rho)$ et $(\mathcal{M}', \rho')$ sont dites \'equivalentes s'il existe un (iso)morphisme de $K$-gerbes $\alpha : \mathcal{M} \to \mathcal{M}'$ tel que $\rho = \rho' \circ \underline{\textup{lien}}(\alpha)$.

On note $\Ger(K,L)$ l'ensemble des classes d'\'equivalence de $K$-gerbes alg\'ebriques li\'ees par $L$. La classe d'une gerbe $\mathcal{M}$ est dite neutre si $\mathcal{M}(K) \neq \emptyset$. Le faisceau $\underline{\aut}(m)$ pour $m\in\cal M(K)$ est dans ce cas repr\'esent\'e par un $K$-groupe alg\'ebrique $G$ qui est une $K$-forme du groupe $\bar G$ sous-jacent \`a $L$. On note alors $n(G)$ la classe de $\cal M$.

\paragraph{Comparaison II.} L'\'enonc\'e suivant affirme que les trois d\'efinitions pr\'ec\'edentes sont \'equivalentes :

\begin{pro}
Soit $L$ un $K$-lien. Alors on a des bijections canoniques et fonctorielles :
\[H^2(K, L) \xrightarrow{\sim} \ext(\Gamma_K, L) \xrightarrow{\sim} \Ger(K,L) \, ,\]
qui font correspondre les sous-ensembles de classes neutres de chacun des trois ensembles. 
\end{pro}

\begin{proof}
On rappelle seulement la construction des diff\'erentes bijections.

\paragraph{Cocycles vers extensions.}
\'Etant donn\'e un $2$-cocycle $(f,u)$, on construit le groupe $E := \bar G (\bar K) \times \Gamma_K$ avec le produit tordu suivant :
\[(g,\sigma) \cdot (h, \tau) := (g \cdot f_\sigma(h) \cdot u_{\sigma, \tau}^{-1}, \sigma \tau) \, ,\]
et on v\'erifie qu'il s'ins\`ere dans une suite exacte courte
\[1 \to \bar G(\bar K) \to E \to \Gamma_K \to 1 \, ,\]
dont le lien associ\'e est clairement $L$, ce qui d\'efinit la fl\`eche $H^2(K, L) \to \ext(\Gamma_K, L)$.

\paragraph{Extensions vers gerbes}
Soit une extension de groupes topologiques
\[1 \to \bar G(\bar K) \to E \xrightarrow{\pi} \Gamma_K \to 1 \, .\]

Soit $A$ une $K$-alg\`ebre \'etale, dont on note $X_A := \hom_{K-\textup{alg}}(A, \bar K)$ le $\Gamma_K$-ensemble fini associ\'e.
On d\'efinit la cat\'egorie $\mathcal{M}_E(A)$ dont les objets sont les paires $(Y, p)$ form\'ees d'un ensemble discret $Y$ muni d'une action continue de $E$ et d'un morphisme $p : Y \to X_A$ qui est $E$-\'equivariant (via le morphisme $\pi$) et tel que l'action de $E$ sur $Y$ induise sur $Y$ une structure de $X_A$-torseur sous $\bar G(\bar K)$ (i.e. $\bar G (\bar K)$ agit simplement transitivement sur les fibres de $p$). Un morphisme $(Y,p) \to (Y',p')$ dans la cat\'egorie $\mathcal{M}_E(A)$ est une application $E$-\'equivariante $\varphi : Y \to Y'$ au-dessus de $X_A$.

On a ainsi d\'efini une cat\'egorie fibr\'ee en groupo\"ides sur le petit site \'etale de $\spec(K)$. 

Les topologies des groupes concern\'es (discr\`ete pour $\bar G(\bar K)$, profinie pour $\Gamma_{K}$) et les propri\'et\'es des morphismes (ouverts et continus) nous disent que l'extension $E$ est localement scind\'ee : il existe une extension finie $K'$ de $K$ (contenue dans $\bar K$) telle que l'extension $E'$, obtenue en tirant $E$ en arri\`ere par le morphisme $\Gamma_{K'} \to \Gamma_K$, est scind\'ee. Une section $s : \Gamma_{K'} \to E'$ permet de d\'efinir un $\bar G (\bar K)$-torseur $p_s : Y_s := E/s(\Gamma_{K'}) \to \Gamma_K/\Gamma_{K'} \xrightarrow{\sim} X_{K'}$, de sorte que $(Y_s,p_s)$ est un objet de $\mathcal{M}_E(K')$. En particulier, $\mathcal{M}_E(K') \neq \emptyset$. Par un argument de continuit\'e, on montre \'egalement que deux objets de $\mathcal{M}_E(K')$ sont localement isomorphes, pour toute extension finie $K'$ de $K$ (cela r\'esulte du fait que deux sections de $E$ sont localement conjugu\'ees par un \'el\'ement de $\bar G (\bar K)$).

V\'erifions que $\mathcal{M}_E$ est un champ. Soit $K'/K$ une extension finie de corps et $K''/K'$ une extension finie galoisienne. On fixe un plongement $K'' \subset \bar K$. Notons $\Gamma := \gal(K''/K')$. Le groupe $\Gamma$ agit naturellement (\`a droite) sur $X_{K''}$ via la formule $\phi \cdot \sigma := \phi \circ \sigma$ pour $\sigma \in \Gamma$ et $\phi \in X_{K''}$. Ainsi, pour tout $\sigma \in \Gamma$, on peut expliciter le foncteur image inverse $\mathcal{M}_E(K'') \to \mathcal{M}_E(K'')$ induit par $\sigma$, de la fa\c con suivante : l'image d'un objet $(Y,p)$ de $\mathcal{M}_E(K'')$ est l'objet $(Y^\sigma, p_\sigma) \in \mathcal{M}_E(K'')$ d\'efini par le diagramme cart\'esien suivant :
\[
\xymatrix{
Y^{\sigma} \ar[r] \ar[d]^{p_\sigma} & Y \ar[d]^p \\
X_{K''} \ar[r]^{\cdot \sigma} & X_{K''} \, , 
}
\]
et on dispose de la description analogue pour les morphismes. On en d\'eduit qu'une donn\'ee de descente sur $(Y,p)$ relativement \`a l'extension $K''/K'$ \'equivaut \`a la donn\'ee, pour tout $\sigma  \in \Gamma$, d'un (iso)morphisme $f_{\sigma} : (Y^\sigma,p_\sigma) \to (Y,p)$, tel que pour tous $\sigma, \tau \in \Gamma$, $f_{\sigma \tau} = f_{\tau} \circ \tau^*(f_\sigma)$. Par cons\'equent, cela \'equivaut \`a la donn\'ee, pour tout $\sigma \in \Gamma$, d'une application $E$-\'equivariante $h_\sigma : Y \to Y$ s'ins\'erant dans le diagramme commutatif suivant :
\[
\xymatrix{
Y \ar[r]^{h_\sigma} \ar[d]^p & Y \ar[d]^p \\
X_{K''} \ar[r]^{\cdot \sigma^{-1}} & X_{K''} \, ,
}
\]
de sorte que $h_{\sigma \tau} = h_\sigma \circ h_\tau$. Cela d\'efinit donc une action (\`a gauche) de $\Gamma$ sur $Y$ compatible avec l'action (\`a gauche) de $\Gamma$ sur $X_{K''}$ (via $\sigma \cdot \phi := \phi \circ \sigma^{-1}$). En outre, on voit que ces actions de $\Gamma$ commutent aux actions de $E$ sur $Y$ et $X_{K''}$. Par cons\'equent, une telle donn\'ee de descente permet de d\'efinir une application $p_0 : Y_0 := Y/\Gamma \to X_{K''}/\Gamma = X_{K'}$ $E$-\'equivariante, et un \'el\'ement $(Y_0, p_0) \in \mathcal{M}_E(K')$ muni d'un isomorphisme entre son image dans $\mathcal{M}_E(K'')$ et $(Y,p)$, compatible \`a la donn\'ee de descente. Par cons\'equent, les donn\'ees de descente dans $\mathcal{M}_E$ sont effectives. Un raisonnement analogue assure que les donn\'ees de descente pour les morphismes sont \'egalement effectives. Donc $\mathcal{M}_E$ est un $K$-champ, localement non vide et localement connexe, donc c'est une $K$-gerbe.

Enfin, si $K' \subset \bar K$, \'etant donn\'e un objet $(Y,p) \in \mathcal{M}_E(K')$, on observe que le choix d'un point $y \in Y$ au-dessus du point naturel de $X_{K'}$ fixe un isomorphisme entre le $L$-faisceau en groupes $\underline{\aut}_L(Y,p)$ et $\bar G(\bar K) \cap N_E(\underline{\textup{Stab}}_E(y_L))$. On en d\'eduit alors un isomorphisme canonique de $K$-liens (ind\'ependant de $(Y,p)$ et de $y$) entre $L_E$ et $L$.

\paragraph{Gerbes vers cocycles}
Soit $\mathcal{M}$ une $K$-gerbe alg\'ebrique et $\tau : L(\mathcal{M}) \to L$ un isomorphisme. Posons alors $L = (\bar G, \kappa)$. Choisissons une section $m \in \mathcal{M}(K')$ pour une extension finie galoisienne $K'/K$. Le morphisme $\tau$ induit un isomorphisme de $\bar K$-groupes $\bar \tau : \underline{\aut}(m) \xrightarrow{\sim} {\bar G}$, bien d\'efini modulo $\textup{Int}(\bar G)$. Suivant la construction du lien $L(\cal{M})$, on obtient pour $\sigma \in \Gamma_K$ un isomorphisme $\varphi_\sigma : {^\sigma m} \to m$ dans la cat\'egorie $\mathcal{M}(K'')$, o\`u $K''$ est une extension finie de $K'$, et l'isomorphisme naturel $\asd{\sigma}{}{}{}{\bar \tau} : \underline{\aut}(\asd{\sigma}{}{}{}{m})\cong\sigma_*\bar G$ nous permet par ailleurs de d\'efinir un morphisme de $\bar K$-groupes $f_\sigma : \sigma_* \bar G \to \bar G$ par $f_\sigma(\alpha) := \varphi_\sigma \circ \alpha \circ {\varphi_\sigma}^{-1}$. Le morphisme $f_\sigma$ est alors dans $\saut(\bar G)$ et, puisque c'est avec ces morphismes que l'on construit le lien associ\'e \`a $\cal M$, son image dans $\sout(\bar G)$ est exactement $\kappa_\sigma$.  En outre, le fait que les morphismes $\varphi_\sigma$ soient d\'efinis sur $K''$ assure que l'application $\sigma \mapsto f_\sigma$ est continue. Si l'on note alors $\asd{\sigma}{}{}{\tau}{\varphi}$ l'image du morphisme $\varphi_\tau$ par le foncteur induit par $\sigma$ et que l'on pose $u_{\sigma, \tau} := \varphi_{\sigma \tau} \circ {^\sigma \varphi_\tau}^{-1} \circ \varphi_\sigma^{-1} \in \underline{\aut}(m)(\bar K) = {\bar G}(\bar K)$ (identification via $\bar \tau$), un calcul simple assure que l'on a les relations suivantes :
\begin{align*}
f_{\sigma \tau} = \int(u_{\sigma, \tau}) \circ f_\sigma \circ f_\tau\qquad & \text{dans }\,\saut(\bar G),\\
u_{\sigma, \tau \upsilon} \cdot f_\sigma(u_{\tau, \upsilon}) = u_{\sigma \tau, \upsilon} \cdot u_{\sigma, \tau}\qquad &\text{dans }\,{\bar G}(\bar K).
\end{align*}
Par cons\'equent $(f,u) \in Z^2(K, L)$. En outre, on v\'erifie que des choix diff\'erents de $\bar \tau$ et $m \in \mathcal{M}(\bar K)$ d\'efinissent des cocycles \'equivalents \`a $(f,u)$.
\end{proof}

\begin{rem}
La correspondance entre gerbes et cocycles est d\'etaill\'ee dans \cite{Br}. Pour le sens ``gerbes vers cocycles'', cf.~les sections 2.2 \`a 2.4; pour le sens ``cocycles vers gerbes'', cf.~2.6 \`a 2.8. 
\end{rem}

\subsubsection{Fonctorialit\'e}\label{section fonct}
Soient $L=(\bar G,\kappa)$ et $L'=(\bar G',\kappa')$ des $K$-liens et soit $\varphi:L\to L'$ un morphisme de $K$-liens. \`A la diff\'erence du cas classique, un tel morphisme n'induit pas forc\'ement un morphisme entre les ensembles de 2-cohomologie non-ab\'elienne. Cependant, on a toujours une relation, not\'ee
\[\varphi^{(2)}:H^2(K,L)\multimap H^2(K,L'),\]
et qui est d\'efinie comme suit. On dit que deux classes $\eta\in H^2(K,L)$ et $\eta'\in H^2(K,L')$ sont reli\'ees si :

Dans le langage des cocycles, s'il existe des cocycles $(f,u)$ et $(f',u')$ repr\'esentant respectivement $\eta$ et $\eta'$ tels que, pour tout $g\in\bar G(\bar K)$ et pour tout $\sigma,\tau\in\Gamma_K$,
\[f'_\sigma(\varphi(g))=\varphi(f_\sigma(g))\quad\text{et}\quad u'_{\sigma,\tau}=\varphi(u_{\sigma,\tau}).\]

Dans le langage des extensions, s'il existe des extensions $E$ et $E'$ repr\'esentant respectivement $\eta$ et $\eta'$ et un morphisme d'extensions
\[\xymatrix{
1 \ar[r] & \bar G(\bar K) \ar[r] \ar[d]^{\varphi} & E \ar[r] \ar[d] & \Gamma_K \ar[r] \ar@{=}[d] & 1,\\
1 \ar[r] & \bar G'(\bar K) \ar[r] & E' \ar[r] & \Gamma_K \ar[r] & 1.
}\]

Dans le langage des gerbes, s'il existe des gerbes $\cal M$ et $\cal M'$ repr\'esentant respectivement $\eta$ et $\eta'$ et un morphisme de gerbes $\cal M\to\cal M'$ li\'e par $\varphi$.\\

On rappelle qu'une telle relation peut \^etre vide en g\'en\'eral et que, dans le cas particulier o\`u $\varphi$ est surjective ou $\bar G'$ est ab\'elien, elle correspond \`a une application $H^2(K,L)\to H^2(K,L')$. Ceci est un exercice facile \`a v\'erifier par exemple dans le langage des cocycles.\\

Enfin, mentionnons qu'il est facile de d\'eduire de ce qui pr\'ec\`ede comment d\'efinir une relation de restriction $H^2(K,L)\to H^2(K',L)$ pour $L$ un $K$-lien et $K'/K$ une extension.

\subsection{Espaces homog\`enes}\label{section esp hom}

Soit $K$ un corps parfait, $G$ un $K$-groupe alg\'ebrique lin\'eaire lisse et $X$ un espace homog\`ene (\`a droite) de $G$, \`a stabilisateurs lisses. On note $x \in X(\bar K)$ un point g\'eom\'etrique de $X$, dont on note $\bar H$ le stabilisateur.

Une construction due \`a Springer permet d'associer au couple $(X,x)$ un $K$-lien $L_X = (\bar H, \kappa_X)$, ainsi qu'une classe $\eta_X \in H^2(K, L_X)$ (on omet abusivement le choix de $x$ dans la suite), tels qu'il existe un morphisme naturel $\rho:L_X\to \underline{\mathrm{lien}}(G)$ reliant $\eta_X$ \`a $n(G)$ (cf.~\cite[Prop.~1.27]{SpringerH2} ou \cite[IV.5.1.3.1]{Giraud}). Rappelons sa construction avec les diff\'erents points de vue mentionn\'es plus haut.\\

Avec le point de vue de Springer-Borovoi (cf.~\cite[\S7]{Borovoi93} ou \cite[5.1]{FSS}), pour tout $\sigma \in \Gamma$, il existe $g_\sigma \in G(\bar K)$ tel que ${^\sigma x} = x \cdot g_\sigma$, et on peut supposer l'application $\sigma \mapsto g_\sigma$ continue. Alors, si l'on note $\sigma_*$ l'automorphisme $\sigma$-semi-alg\'ebrique de $\bar G$ induit naturellement par $\sigma$ (cf.~\cite[1.4]{Borovoi93}), on voit que l'automorphisme $\int(g_\sigma) \circ \sigma_*$ est aussi $\sigma$-semi-lin\'eaire mais de plus laisse $\bar H$ invariant. On note $f_\sigma$ sa restriction \`a $\bar H$. On a donc une application continue $f : \Gamma_K \to \saut(\bar H)$ qui induit un morphisme continu (donc un $K$-lien) $\kappa_X : \Gamma_K \to \sout(\bar H)$. On note alors $L_X := (\bar H, \kappa_X)$ le lien correspondant. On pose enfin $u_{\sigma, \tau} := g_{\sigma \tau} \cdot \asd{\sigma}{}{}{\tau}{g} \cdot g_\sigma^{-1} \in {\bar H}(\bar K)$, et on v\'erifie que $\eta_X := [(f,u)]$ est un \'el\'ement de $H^2(K,L_X)$ ne d\'ependant pas du choix des $g_\sigma$.

Par construction, on voit que $\eta_X$ est neutre si et seulement s'il existe un $K$-torseur $P$ sous $G$ et un $K$-morphisme $G$-\'equivariant $P \to X$.\\

Avec le point de vue des extensions de groupes (cf.~\cite[5.1]{FSS}), on d\'efinit $E_X$ comme le sous-groupe de $G(\bar K) \rtimes \Gamma_K$ form\'e des \'el\'ements $(g,\sigma)$ tels que ${^\sigma x} = x \cdot g$. Alors on dispose d'une suite exacte naturelle de groupes topologiques
\[1 \to {\bar H}(\bar K) \to E_X \to \Gamma_K \to 1 \, ,\]
telle que le lien associ\'e est exactement $L_X$ et dont la classe dans $H^2(K, L_X)$ est exactement $\eta_X$.\\

Avec le point de vue de Giraud (cf.~\cite[IV.5.1]{Giraud}), on peut associer \`a $X$ la cat\'egorie fibr\'ee $\mathcal{M}_X$ sur le petit site \'etale de $\spec K$ des rel\`evements de $X$ en un torseur sous $G$, i.e. pour toute $K$-alg\`ebre \'etale $A$, la cat\'egorie fibre $\mathcal{M}_X(A)$ a pour objets les couples $(P,\alpha)$ o\`u $P \to \spec A$ est un torseur sous $G$ et $\alpha : P \to X_A$ est un $A$-morphisme $G$-\'equivariant, et pour morphismes les morphismes de $G$-torseurs commutant aux morphismes vers $X_A$. Alors $\mathcal{M}_X$ est une gerbe li\'ee par $L_X$, et sa classe dans $H^2(K, L_X)$ est exactement $\eta_X$. En outre, on dispose d'un morphisme de gerbes naturel $\varphi_X : \mathcal{M}_X \to \TORS(G)$ d\'efini par $(P,\alpha) \mapsto P$.
Pour pr\'eciser la fonctorialit\'e\footnote{Les auteurs remercient le rapporteur de cette suggestion pertinente.} de cette construction, d\'efinissons $\EspHom(K)$ comme la cat\'egorie des paires $(G,X)$ o\`u $G$ est un $K$-groupe alg\'ebrique lin\'eaire lisse et $X$ un $K$-espace homog\`ene de $G$ \`a stabilisateurs lisses, et $\GerbesHom(K)$ comme la cat\'egorie des triplets $(G,\mathcal{M}, \varphi)$ o\`u $G$ est un $K$-groupe alg\'ebrique lin\'eaire lisse, $\mathcal{M}$ une $K$-gerbe alg\'ebrique et $\varphi : \mathcal{M} \to \TORS(G)$ un morphisme injectif de $K$-gerbes. On dispose alors d'un foncteur ``gerbe de Springer'' :
\begin{align*}
\eta \colon \EspHom(K) &\to \GerbesHom(K) \\
(G,X)  &\mapsto (G, \mathcal{M}_X, \varphi_X) \, .
\end{align*}

\section{Gerbes et espaces homog\`enes de $\sl_{n,K}$}\label{section gerbes dans SLn}
Les r\'esultats suivants sont cruciaux pour la construction g\'eom\'etrique utilis\'ee dans la d\'emonstration du th\'eor\`eme principal. Ils g\'en\'eralisent des outils tout \`a fait courants dans le cadre o\`u l'on poss\`ede un point rationnel, comme le plongement d'un $K$-groupe affine donn\'e dans $\sl_{n,K}$, ou le c\'el\`ebre lemme sans nom, qui \'etablit en particulier la stable birationalit\'e des quotients correspondants \`a deux tels plongements.

\subsection{``Plonger une $K$-gerbe affine dans $\sl_{n,K}$''}
Soit $K$ un corps parfait et $\bar K$ une cl\^oture alg\'ebrique de $K$. \'Etant donn\'e un $K$-groupe $G$, un plongement de $G$ dans $\sl_{n,K}$ permet de construire un espace homog\`ene $X=G\backslash\sl_{n,K}$ muni d'un $K$-point \`a stabilisateur $G$. Inversement, \'etant donn\'e un espace homog\`ene de $\sl_{n,K}$ muni d'un $K$-point, on obtient un $K$-groupe $G$ plong\'e dans $\sl_{n,K}$ en prenant le stabilisateur du point. Dans cette section on \'etablit une variante de cette \'equivalence dans le cadre des espaces homog\`enes sans point rationnel. Un tel espace nous donne toujours une gerbe qui correspond \`a sa classe de Springer (cf.~la section \ref{section esp hom}). Par ``plonger une $K$-gerbe dans $\sl_{n,K}$'', on veut justement \'evoquer la construction d'un espace homog\`ene de $\sl_{n,K}$ dont la classe de Springer soit la (classe d'\'equivalence de la) gerbe que l'on s'est donn\'ee.

On utilise ci-dessous les notations donn\'ees en section \ref{section H2 nonab}

\begin{pro} \label{prop compatibilite des liens avec sln}
Soit $\mathcal{M}$ une $K$-gerbe alg\'ebrique. Alors il existe un entier $n \in \n$ et un morphisme injectif de $K$-gerbes $\rho : \mathcal{M} \to \TORS(\sl_{n,K})$.
\end{pro}

\begin{proof}
Soit $K'/K$ une extension finie galoisienne telle que $\mathcal{M}(K') \neq \emptyset$. Choisissons $m \in \mathcal{M}(K')$ et notons $G' := \underline{\textup{Aut}}(m)$ le $K'$-groupe lin\'eaire associ\'e. On a alors une \'equivalence de $K'$-gerbes $\varphi:\mathcal{M}_{K'} \xrightarrow{\sim} \underline{\mathrm{TORS}}(G')$ (cf.~\cite[III.2.5.1]{Giraud}).

Posons $H:= R_{K'/K}(G')$. La restriction \`a la Weil permet alors de d\'efinir un morphisme injectif de $K$-gerbes $\epsilon : \mathcal{M} \to \TORS(H)$ de la fa\c con suivante : soit $S$ un $K$-sch\'ema \'etale. Pour tout $m \in \mathcal{M}(S)$, on note $\epsilon(m) \in \TORS(H)(S)$ le $S$-torseur sous $H$ d\'efini par $R_{S_{K'}/S}(\varphi(m_{K'}))$ (par d\'efinition, $\varphi(m_{K'})$ est un $S_{K'}$-torseur sous $G'$). De m\^eme, si $\phi : m \to m'$ est un morphisme dans $\mathcal{M}(S)$, on d\'efinit le morphisme $\epsilon(\phi) : \epsilon(m) \to \epsilon(m')$ dans $\TORS(H)(S)$ comme le morphisme $R_{S_{K'}/S}( \phi_{K'} : m_{K'} \to m'_{K'})$.	

Pour finir, il existe un morphisme injectif de $K$-groupes affines $j : H\to \sl_{n,K}$ et on note 
\[\rho : \mathcal{M} \xrightarrow{\epsilon} \TORS(H) \xrightarrow{j} \TORS(\sl_{n,K}),\]
le morphisme compos\'e de $K$-gerbes.
\end{proof}

\begin{pro} \label{prop existence des espaces homogenes}
Le foncteur ``classe de Springer'' $\eta : \EspHom(K) \to \GerbesHom(K)$ d\'efini dans la section \ref{section esp hom} est une \'equivalence de cat\'egories.
\end{pro}

\begin{proof}
Construisons un quasi-inverse explicite $\mu : \GerbesHom(K) \to \EspHom(K)$. Soit $E := (G, \mathcal{M}, \rho) \in \GerbesHom(K)$. On consid\`ere le foncteur $Y_E$ qui \`a tout morphisme \'etale $S \to \spec(K)$ associe l'ensemble $Y_E(S)$ des couples $(m, \varphi)$, avec $m \in \mathcal{M}(S)$ et $\varphi : \rho(m) \xrightarrow{\sim} G_S$ un morphisme dans $\TORS(G)(S)$. On dit que deux \'el\'ements $(m, \varphi)$ et $(m', \varphi')$ sont \'equivalents s'il existe un (iso)morphisme $\psi : m \to m'$ dans $\mathcal{M}(S)$ tel que $\varphi' \circ \rho(\psi) = \varphi$. On d\'efinit $X_E(S)$ comme le quotient de $Y_E(S)$ par cette relation d'\'equivalence et on note $\overline{(m,\varphi)}$ la classe d'\'equivalence de $\varphi$.  On d\'efinit ainsi un foncteur $X_E$ de la cat\'egorie des $K$-sch\'emas \'etale vers la cat\'egorie des ensembles. On v\'erifie facilement que $X_E$ est un faisceau \'etale, admettant des sections localement pour la topologie \'etale.

On d\'efinit une action de $G$ sur $X_E$ de la fa\c con suivante : pour tout $S$ \'etale sur $K$, pour tout $g \in G(S)$ et $\overline{(m, \varphi)} \in X_E(S)$, on note $\overline{(m, \varphi)} \cdot g := \overline{(m, g \circ \varphi)}$, o\`u $g : G_S \to G_S$ d\'esigne la translation \`a gauche par $g$.

On v\'erifie que $X_E$ est un espace homog\`ene de $G$ (comme faisceau \'etale). Il est clair que la construction de $X_E$ est fonctorielle en $E$.

Enfin, montrons que $X_E$ est repr\'esentable par une vari\'et\'e alg\'ebrique, i.e. que $(G,X_E) \in \EspHom(K)$. Choisissant $x = \overline{(m,\varphi)} \in X_E(K')$, pour une extension finie galoisienne $K'$ de $K$, on note $H' := \underline{\textup{Aut}}(m)$, qui n'est autre que le stabilisateur de $x$ dans $G_{K'}$. Alors $H'$ est un $K'$-groupe lin\'eaire lisse, et $\rho$ et $\varphi$ r\'ealisent $H'$ comme un sous-groupe alg\'ebrique (ferm\'e) de $G_{K'}$. Alors le morphisme naturel $G_{K'} \to X_{E,K'}$ donn\'e par l'action de $G$ sur $x$ induit un isomorphisme $H' \backslash G_{K'} \xrightarrow{\sim} X_{E,K'}$, ce qui assure que $X_{E,K'}$ est repr\'esentable par une $K'$-vari\'et\'e quasi-projective. En outre, le choix, pour tout $\sigma \in \Gamma_K$, d'un isomorphisme $\varphi_{\sigma} : {^\sigma m} \to m$ dans $\mathcal{M}(K'')$ (avec $K''$ extension finie de $K'$) permet de construire un $2$-cocycle $(f,u) \in Z^2(K, L)$, o\`u $L$ est le $K$-lien de $\mathcal{M}$, et une application continue $c : \Gamma_K \to G(\bar{K})$ telle que si $f' : \Gamma_K \to \saut(\bar{G})$ est le morphisme induit par l'action naturelle, on a $f_\sigma = \int(c_{\sigma}) \circ f'_\sigma$ et $u_{\sigma, \tau} = c_{\sigma \tau} f'_\sigma(c_\tau)^{-1} c_\sigma^{-1}$. 

On munit alors la $\bar K$-vari\'et\'e quasi-projective $\bar{X} := \bar{H'} \backslash \bar{G}$, point\'ee par la classe $x_0$ de $\bar{H'}$, de l'action de Galois suivante : pour tout $\sigma \in \Gamma_K$, on pose ${^\sigma (x_0 \cdot g)} := x_0 \cdot (c_\sigma f'_\sigma(g))$. On v\'erifie que cela d\'efinit bien une action du groupe de Galois de $K$, et la descente galoisienne (cf.~par exemple \cite[\S 6.2, Ex.~B]{BLR}) assure qu'il existe une $K$-forme $X$ de $\bar{X}$ munie d'une action naturelle de $G$ d\'efinie sur $K$. Enfin, on v\'erifie que le morphisme pr\'ec\'edent $\bar{X} \xrightarrow{\sim} \bar X_E$ est compatible avec les actions de $\Gamma_K$ des deux c\^ot\'es. Cela assure que $X_E$ est repr\'esentable (par $X$).

On a ainsi d\'efini un foncteur
\begin{align*}
\mu:\GerbesHom(K) &\to \EspHom(K)\\
E := (G, \mathcal{M}, \rho) &\mapsto (G,X_E).
\end{align*}
Montrons maintenant que $\eta$ et $\mu$ sont quasi-inverses.

Soit $(G,X) \in \EspHom(K)$ et notons $E_X := (G,\mathcal{M}_X,\rho_X) = \eta(X)$. Pour tout $K$-sch\'ema \'etale $S$, on dispose d'une application canonique $X_{E_X}(S) \to X(S)$ d\'efinie par $\overline{(m,\varphi)} \mapsto \pi_m(\varphi^{-1}(1_G))$, o\`u $\pi_m : m \to X_S$ est l'application naturelle. Par construction, cette application est une bijection $G(S)$-\'equivariante, fonctorielle en $S$ et $X$. D'o\`u une transformation naturelle $\mu \circ \eta \to \id$.

Soit $E=(G, \mathcal{M}, \rho) \in \GerbesHom(K)$ et notons $(G,X_E) := \mu(E)$. Pour tout $K$-sch\'ema \'etale $S$, on dispose d'un foncteur naturel $\mathcal{M}(S) \to \mathcal{M}_{X_E}(S)$ d\'efini sur les objets par $m \mapsto (\varpi_m : \underline{\textup{Isom}}_S(\rho(m),G_S) \to (X_E)_S)$, avec $\varpi_m(\varphi) := \overline{(m, \varphi)}$. On v\'erifie que ce foncteur induit une \'equivalence de $K$-gerbes $\mathcal{M} \to \mathcal{M}_{X_E}$ compatible aux morphismes de gerbes $\rho$ et $\rho_{X_E}$, fonctorielle en $E$. D'o\`u une transformation naturelle $\eta \circ \mu \to \id$.
\end{proof}

\begin{cor} \label{corollaire plongement de liens}
Soit $\mathcal{M}$ une $K$-gerbe alg\'ebrique. Alors il existe un entier $n$ et un $K$-espace homog\`ene $X$ de $\sl_{n,K}$ dont la gerbe de Springer $\mathcal{M}_X$ est isomorphe \`a $\mathcal{M}$.
\end{cor}

\begin{proof}
Il suffit d'appliquer la proposition \ref{prop compatibilite des liens avec sln} et la proposition \ref{prop existence des espaces homogenes}.
\end{proof}

\subsection{Le lemme sans nom ni point rationnel}
Le lemme sans nom (cf. \cite[\S3.2]{ColliotSansucChili}) permet de montrer que, \'etant donn\'e un $K$-groupe alg\'ebrique $G$, tous les espaces homog\`enes $G\backslash\sl_n$ ou $G\backslash\gl_n$ (qui poss\`edent tous des points rationnels) sont $K$-stablement birationnels entre eux, ind\'ependamment du plongement de $G$ dans ces groupes. Ayant donn\'e un sens au plongement d'une gerbe dans $\sl_n$ ou $\gl_n$, on peut maintenant se poser la m\^eme question pour des espaces homog\`enes sans point rationnel, mais ayant la m\^eme gerbe associ\'ee. Le th\'eor\`eme suivant fournit les outils n\'ecessaires pour r\'epondre affirmativement \`a cette question, mais aussi pour les constructions g\'eom\'etriques de la preuve du th\'eor\`eme principal.

\begin{thm}\label{thm lemme sans nom}
Soient $G$ un $K$-groupe lin\'eaire lisse et $X$ un espace homog\`ene de $G$. Soient $N \subset G$ un sous-groupe distingu\'e et $(G',\mathcal{M}', \rho') \in \GerbesHom(K)$. On suppose qu'il existe un diagramme commutatif de $K$-gerbes
\[\xymatrix{
\mathcal{M}_X \ar[r]^{\varpi} \ar[d]^{\rho_X} & \mathcal{M}' \ar[d]^{\tau} \\
\TORS(G) \ar[r] & \TORS(G/N) \, ,
}\]
avec $\varpi$ surjectif et l'on note $\rho'_X:=\rho'\circ\varpi$.

Alors il existe un $K$-espace homog\`ene $X'$ de $(G/N \times G')$ tel que $\eta(X') = (G/N \times G',\mathcal{M}',\tau \times \rho')$, un $K$-espace homog\`ene $Y$ de $(G \times G')$ tel que $\eta(Y) = (G \times G',\mathcal{M}_X,\rho_X \times \rho'_X)$, et des morphismes naturels de $K$-espaces homog\`enes 
\[\xymatrix{
& Y \ar[rd]^q \ar[ld]_p & \\
X & & X' \, ,
}\]
de sorte que $Y$ est un $X$-torseur sous $G'$ et $Y$ est un $X'$-espace homog\`ene de $N_{X'}$, dont les fibres g\'eom\'etriques ont un stabilisateur isomorphe \`a $\ker \bar{\varpi} \subset \bar N$, o\`u $\bar{\varpi}$ est un morphisme de $\bar K$-groupes alg\'ebriques repr\'esentant le morphisme de liens $L(\mathcal{M}_X) \to L(\mathcal{M}')$ liant $\varpi$.
\end{thm}

\begin{rem}
On appliquera souvent le th\'eor\`eme au cas $G'=\sl_{n,K}$, en utilisant le morphisme $\rho'$ fourni par la proposition \ref{prop compatibilite des liens avec sln}.
\end{rem}

\begin{proof}
Les hypoth\`eses assurent l'existence d'un diagramme commutatif de $K$-gerbes :
\[\xymatrix{
& \mathcal{M}_X \ar@{->>}[rd]^{\varpi} \ar@{^{(}->}[d]^{\rho \times \rho'_X} \ar@{=}[ld] & \\
\mathcal{M}_X \ar@{^{(}->}[d]_{\rho_X} & \TORS(G \times G') \ar@{->>}[dl]^{p_1} \ar@{->>}[rd]_{q_1} & \mathcal{M}' \ar@{^{(}->}[d]^{\tau \times \rho'} \\
\TORS(G) & & \TORS(G/N \times G') \, .
}\]
On applique alors la proposition \ref{prop existence des espaces homogenes}, qui fournit un diagramme de $K$-espaces homog\`enes
\[
\xymatrix@R=1em{
& Y \ar[rd]^q \ar[ld]_p & \\
X & & X' \, ,
}
\]
compatible avec les morphismes naturels de $K$-groupes alg\'ebriques $p_1:G\times G'\to G$ et $q_1:G\times G'\to G/N\times G$. Les propri\'et\'es de $p$ et $q$ se d\'eduisent par changement de base de $K$ \`a $\bar K$, o\'u les gerbes sont trivialis\'ees.
\end{proof}

On obtient une premi\`ere application int\'eressante en posant $G'=\sl_{n,K}$ :

\begin{cor}[Le lemme sans nom ni point rationnel]\label{corollaire lemme sans nom}
Soit $G$ un $K$-groupe alg\'ebrique sp\'ecial $K$-rationnel et soit $X$ un espace homog\`ene de $G$ de gerbe de Springer $\rho_X : \mathcal{M}_X \to \TORS(G)$. Alors tout espace homog\`ene $X'$ de $\sl_{n,K}$ dont la gerbe de Springer $\mathcal{M}_{X'}$ est isomorphe \`a $\mathcal{M}_X$ (donn\'e par exemple par le corollaire \ref{corollaire plongement de liens}), est $K$-stablement birationnel \`a $X$.

En particulier, deux espaces homog\`enes de $\sl_{n,K}$, dont les gerbes associ\'ees sont isomorphes, sont $K$-stablement birationnels.
\end{cor}

\begin{proof}
Il suffit d'appliquer le th\'eor\`eme \ref{thm lemme sans nom} \`a $\mathcal{M}' := \mathcal{M}_{X'}$, $\varpi : \mathcal{M}_X \to \mathcal{M}_{X'}$ un isomorphisme, $\rho' := \rho_{X'} : \mathcal{M}_{X'} \to \TORS(\sl_{n,K})$, $N=G$. On obtient alors le diagramme suivant d'espaces homog\`enes ($Y$ est un espace homog\`ene de $G \times \sl_{n,K}$) :
\[\xymatrix{
& Y \ar[rd]^q \ar[ld]_p & \\
X & & X' \, ,
}\]
o\`u $p$ et $q$ sont des torseurs sous $\sl_{n,K}$ et $G$ respectivement. Puisque tous les deux sont des $K$-groupes sp\'eciaux $K$-rationnels, on trouve que tant $X$ que $X'$ sont $K$-stablement birationnels \`a $Y$, ce qui conclut.
\end{proof}

\begin{rems}
{\bf 1.} Les \'enonc\'es des corollaires \ref{corollaire plongement de liens} et \ref{corollaire lemme sans nom} sont valables aussi en rempla\c cant $\sl_{n,K}$ par $\gl_{n,K}$.\\
{\bf 2.} La proposition \ref{prop compatibilite des liens avec sln} et le corollaire \ref{corollaire plongement de liens} permettent d'associer \`a toute $K$-gerbe alg\'ebrique une classe d'\'equivalence birationnelle stable de $K$-espaces homog\`enes de groupes sp\'eciaux $K$-rationnels. De cette fa\c con, toute propri\'et\'e de ces espaces homog\`enes invariante par \'equivalence birationnelle stable peut \^etre vue comme une propri\'et\'e des gerbes alg\'ebriques.
\end{rems}

\section{``L'action d'une gerbe sur un tore''}\label{section ono}
Dans \cite{GLA-BMWA}, on a donn\'e une petite g\'en\'eralisation du Lemme d'Ono au cadre des tores avec une $K$-action d'un $K$-groupe fini et lisse. Cela \'etait suffisant pour traiter l'approximation faible pour les espaces homog\`enes, mais pour traiter le principe de Hasse, il faut g\'en\'eraliser encore un peu cette notion.

Soit donc $K$ un corps parfait et $\mathcal{M}$ une $K$-gerbe alg\'ebrique, de lien $L$. On note $\bar G$ le $\bar K$-groupe alg\'ebrique sous-jacent \`a $L$, et $E$ une extension de $\Gamma_K$ par $\bar G (\bar K)$ associ\'ee \`a $\mathcal{M}$. Consid\'erons le morphisme surjectif $\bar G\to\bar G^\f$ (cf.~les notations en d\'ebut de section \ref{section prel}). Puisque $\bar G^\circ$ est un sous-groupe caract\'eristique, ce morphisme induit un $K$-lien naturel $L^\f$ sur $\bar G^\f$ et un morphisme de liens surjectif $\pi^\f:L\to L^\f$. On dispose alors de la gerbe $\mathcal{M}^\f$ comme le pouss\'e en avant (cf.~\cite[IV.2.3.18]{Giraud}) de $\mathcal{M}$ par $\pi^\f : L \to L^\f$. Elle correspond \`a l'extension $E^\f$ de $\Gamma_K$ par $\bar G^\f(\bar K)$ (qui est de plus un groupe profini) obtenue \`a partir de $E$. On fixe ces donn\'ees dans toute cette section.

La notion d'``action d'une gerbe sur un tore'' sera r\'ealis\'ee par l'action de l'extension correspondante selon la d\'efinition suivante:

\begin{defi}\label{defi eta-tore}
On d\'efinit un \emph{$E$-tore} comme la donn\'ee d'un $\bar K$-tore $\bar T$ et d'un homomorphisme continu $\phi:E^\f\to \aut_{\bar K\textup{-gr}}(\bar T)$.

Un morphisme de $E$-tores, ou \emph{$E$-morphisme}, est un morphisme de $\bar K$-tores compatible avec les actions respectives de $E^\f$. En particulier, on appelera \emph{$E$-isog\'enie} un $E$-morphisme surjectif et \`a noyau fini.
\end{defi}

L'extension $E^\f$ correspond au quotient de $E$ par le sous-groupe distingu\'e $\bar G^\circ(\bar K)$. De la d\'efinition ci-dessus on obtient une action naturelle de $E$ sur $\bar T$. D'autre part, on obtient aussi tout naturellement une action de $E$ (via $E^\f$) sur le module de caract\`eres $\hat{\bar T}=\hom_{\bar K}(\bar T,\gm)$ et il est facile de voir qu'une telle action d\'efinit en retour une structure de $E$-tore sur $\bar T$. On obtient ainsi un foncteur contravariant qui \`a un $E$-tore associe son module des caract\`eres.

On dira qu'un $E$-tore est \emph{trivial} si l'action de $E^\f$ sur $\hat{\bar T}$ l'est et \emph{quasi-trivial} si cette action fait de $\hat{\bar T}$ un $E^\f$-module de permutation (notons que tant l'action sur $\hat{\bar T}$ que celle sur $\bar T$ se factorisent par un quotient fini de $E^\f$ par continuit\'e).

On voit facilement aussi que les faits classiques sur les isog\'enies de tores restent valables dans ce contexte. Notamment, on peut dire que deux $E$-tores $\bar T$ et $\bar T'$ sont \emph{$E$-isog\`enes} s'il existe une $E$-isog\'enie de l'un sur l'autre. L'existence d'une telle isog\'enie implique en effet imm\'ediatement celle d'une $E$-isog\'enie dans l'autre sens.\\

Notons qu'il y a une fa\c con naturelle d'\'etendre le morphisme $\phi:E^\f\to \aut_{\bar K}(\bar T)$ en un morphisme continu $\phi':E^\f\to \saut(\bar T)$. En effet, le groupe $\saut_{\bar K}(\bar T)$ est canoniquement isomorphe au produit direct $\aut_{\bar K}(\bar T)\times\Gamma_K$ via la $K$-forme d\'eploy\'ee de $\bar T$. Il suffit alors de coupler $\phi$ avec la projection naturelle $E^\f\to\Gamma_K$ pour obtenir le morphisme $\phi'$ voulu. Ceci nous dit qu'on aurait pu d\'efinir un $E$-tore de fa\c con \'equivalente comme la donn\'ee d'un $\bar K$-tore $\bar T$ et d'un morphisme continu $\phi':E^\f\to \saut(\bar T)$ s'ins\'erant dans le carr\'e commutatif suivant :
\[\xymatrix{
E^\f \ar@{->>}[r] \ar[d]_{\phi'} & \Gamma_K \ar@{=}[d] \\
\saut(\bar T) \ar@{->>}[r] & \Gamma_K.
}\]
Au vu de la discussion pr\'ec\'edente, cette nouvelle d\'efinition est essentiellement \'equivalente \`a la pr\'ec\'edente, et se pr\^ete plus facilement \`a des g\'en\'eralisations non commutatives.

\begin{rem}
Ceci est une g\'en\'eralisation de la notion de $K$-action d'un $K$-groupe alg\'ebrique $G$ sur un $K$-tore $T$. En effet, une telle action se factorise toujours par le quotient $G^\f$ de $G$ et celle-ci donne de fa\c con naturelle un $E_G$-tore, o\`u $E_G$ d\'esigne l'extension naturelle $G(\bar{K}) \rtimes \Gamma_K$. L'action de $G^\f$ correspond \`a un morphisme $G^\f(\bar K)\to\aut_{\bar K}(\bar T)$, tandis que la structure de $K$-groupe de $T$ donne un morphisme continu $\Gamma_K\to\aut_{\bar K}(\bar T)$. Le fait que l'action de $G^\f$ soit d\'efinie sur $K$ permet alors d'\'etendre naturellement ces deux fl\`eches en un morphisme $G^\f(\bar K)\rtimes\Gamma_K\to\aut_{\bar K}(\bar T)$, donnant ainsi un $E_{G^\f}$-tore et donc un $E_G$-tore. On g\'en\'eralise ainsi ce qui a \'et\'e fait dans \cite{GLA-BMWA} autour du lemme d'Ono pour traiter les questions d'approximation faible.
\end{rem}

Le lemme suivant est une cons\'equence imm\'ediate de \cite[Lem.~9]{GLA-BMWA} (r\'esultat sur les modules de type fini sans torsion, d\^u \`a Ono) et des paragraphes qui pr\'ec\`edent.

\begin{lem}\label{lemme Ono}
Soient $K$ un corps parfait, $\mathcal{M}$ une $K$-gerbe alg\'ebrique correspondant \`a une extension $E$. Soit $(\bar T,\phi)$ un $E$-tore et soit $E'$ le sous-groupe ouvert de $E^\f$ correspondant au noyau de $\phi$. Alors il existe des $E$-tores quasi-triviaux $(\bar P,\phi_P)$ et $(\bar Q,\phi_Q)$ et un entier $r\geq 1$ tels que :
\begin{itemize}
\item les morphismes $\phi_P,\phi_Q$ sont triviaux sur $E'$ ;
\item le tore $\bar T^r\times \bar P$ est $E$-isog\`ene \`a $\bar Q$.
\end{itemize}
\end{lem}

Ce r\'esultat sera utilis\'e dans la d\'emonstration du th\'eor\`eme principal. En effet, \'etant donn\'ee une $K$-gerbe alg\'ebrique $\mathcal{M}$ associ\'ee \`a une extension $E$ de $\Gamma_K$ par $\bar G (\bar K)$, on construit un $E$-tore naturel de la fa\c con suivante.\\

Consid\'erons le morphisme surjectif $\bar G\to\bar G^\torf$. Puisque $\bar G^\ssu$ est aussi un sous-groupe caract\'eristique, on obtient un $K$-lien naturel $L^\torf$ sur $\bar G^\torf$ et un morphisme de liens surjectif $\pi^\torf:L\to L^\torf$ qui factorise clairement $\pi^\f$, ce qui nous permet de d\'efinir la gerbe $\mathcal{M}^\torf$ et l'extension $E^\torf$.

Le tore $\bar G^\tor$ admet alors naturellement une structure de $E$-tore. En effet, on dispose de l'extension $E^\torf$ :
\[1\to \bar G^\torf(\bar K)\to E^\torf\to \Gamma_K \to 1 \, .\]
Puisque $\bar G^\tor$ est caract\'eristique dans $\bar G^\torf$, on sait que le sous-groupe $\bar G^\tor(\bar K)$ est distingu\'e dans $E^\torf$. On en d\'eduit la suite exacte
\[1\to \bar G^\tor(\bar K) \to E^\torf\to E^\f\to 1.\]
Le groupe $E^\f$ agit alors de fa\c con naturelle sur $\bar G^\tor(\bar K)$ et il est facile de voir que cette action se rel\`eve en un morphisme continu vers $\saut(\bar G^\tor)$, ce qui munit $\bar G^\tor$ d'une structure de $E$-tore.\\

Notons pour finir que cette construction est fonctorielle en $E$ et en $K$.

\section{\'Enonc\'e et d\'emonstration du th\'eor\`eme principal}\label{section thm principal}
Soit $k$ un corps de nombres. Rappelons que l'on s'int\'eresse dans cet article \`a des propri\'et\'es arithm\'etiques comme \eqref{BMPH} (cf.~la section \ref{section intro}) pour les espaces homog\`enes des groupes lin\'eaires. Comme il a \'et\'e dit dans l'introduction, il nous faut cependant un minimum de propri\'et\'es d'approximation pour appliquer notre r\'esultat principal. On d\'efinit alors l'\emph{approximation r\'eelle} comme la propri\'et\'e
\begin{equation}\label{AR}
X(k)\neq\emptyset\,\Rightarrow\, \overline{X(k)}=\prod_{v\in\infty_k} X(k_v),\tag{AR}
\end{equation}
o\`u $\infty_k$ repr\'esente l'ensemble de places r\'eelles de $k$. C'est le couple \eqref{BMPH}+\eqref{AR} qu'il nous faut consid\'erer, afin que la propri\'et\'e suivante soit v\'erifi\'ee (nous remercions Olivier Wittenberg pour sa suggestion de pr\'esenter ainsi le r\'esultat) :

\begin{defi}
Soit $K$ un corps de caract\'eristique 0 et soit $P$ une propri\'et\'e des $K$-vari\'et\'es. On dit que $P$ est une \emph{bonne propri\'et\'e} si elle v\'erifie, pour $X,Y$  birationnelles \`a des $K$-espaces homog\`enes:
\begin{itemize}
\item[(i)] si $X$ et $Y$ sont stablement birationnelles, alors $P(X)\Leftrightarrow P(Y)$ ;
\item[(ii)] si $X\to Y$ est un morphisme admettant une section, alors $P(X)\Rightarrow P(Y)$ ;
\item[(iii)] si $X\to Y$ est un morphisme surjectif dont toute fibre est un espace homog\`ene d'un groupe $G=G^\ssu$ avec $G^\ss$ simplement connexe, \`a stabilisateur g\'eom\'etrique $\bar H = \bar H^\ssu$, alors $P(Y)\Rightarrow P(X)$.
\end{itemize}
\end{defi}

La justification de cette d\'efinition vient du fait que la preuve du r\'esultat principal de \cite{GLA-BMWA} concernant \eqref{BMAF}, utilise les trois propri\'et\'es ci-dessus (qui sont v\'erifi\'ees pour \eqref{BMAF}, cf.~la proposition \ref{prop P} ci-dessous). Or, la troisi\`eme n'est pas v\'erifi\'ee pour \eqref{BMPH}, ce qui nous a amen\'es \`a rajouter \eqref{AR} dans nos hypoth\`eses. Le r\'esultat principal de ce texte est le suivant, et sa preuve est purement g\'eom\'etrique :

\begin{thm}\label{main thm}
Soit $K$ un corps de caract\'eristique 0, soit $P$ une bonne propri\'et\'e.

Soit $G$ un $K$-groupe lin\'eaire connexe, $X$ un espace homog\`ene de $G$ et $x\in X(\bar K)$ un point \`a stabilisateur g\'eom\'etrique $\bar H$. Alors, il existe un $\bar K$-groupe fini $\bar F$, extension de $\bar H^\f$ par un $\bar K$-groupe ab\'elien, un $K$-espace homog\`ene $X'$ de $\sl_{n,K}$ et $x' \in X'(\bar K)$ \`a stabilisateur g\'eom\'etrique $\bar F$ tels que $P(X')\Rightarrow P(X)$.

En particulier, si tout espace homog\`ene de $\sl_{n,K}$ \`a stabilisateurs finis v\'erifie $P$, il en va de m\^eme pour tout espace homog\`ene d'un groupe lin\'eaire connexe.
\end{thm}

La relation entre $X'$ et $X$ est m\^eme un peu plus forte que celle de l'\'enonc\'e. Voir \`a ce sujet la remarque \`a la fin du texte.

\begin{rem}
On peut aussi obtenir un r\'esultat pour des corps parfaits en caract\'eristique positive si l'on se restreint aux espaces homog\`enes \`a stabilisateurs \emph{lisses}.
\end{rem}

Avant de passer \`a la d\'emonstration du th\'eor\`eme, voyons quelques propri\'et\'es arithm\'etiques qui sont de bonnes propri\'et\'es des $k$-vari\'et\'es pour $k$ un corps de nombres. La derni\`ere correspond \`a notre th\'eor\`eme \ref{thm principal}.

\begin{pro}\label{prop P}
Soit $k$ un corps de nombres. Les propri\'et\'es suivantes sont de bonnes propri\'et\'es des $k$-vari\'et\'es :
\begin{itemize}
\item le principe de Hasse avec approximation r\'eelle ;
\item l'approximation faible ;
\item l'approximation faible hors de $S$ avec $S$ un ensemble de places \emph{non archim\'ediennes} (en particulier l'approximation r\'eelle) ;
\item la propri\'et\'e \eqref{BMPH}{\rm +}\eqref{AR} ;
\item la propri\'et\'e \eqref{BMAF} ;
\item la propri\'et\'e \eqref{BMPH}{\rm +}\eqref{BMAF} ;
\end{itemize}
\end{pro}

\begin{rem}
Dans un travail r\'ecent \cite{HW}, Harpaz et Wittenberg montrent qu'une variante de la conjecture (E) (i.e. de l'exactitude de la suite \eqref{E}) est une bonne propri\'et\'e des $k$-vari\'et\'es, et que cette variante est v\'erifi\'ee par les espaces homog\`enes de $\sl_{n,k}$ \`a stabilisateurs finis. Les auteurs utilisent alors le th\'eor\`eme \ref{main thm} pour obtenir la conjecture (E) pour tout espace homog\`ene d'un groupe lin\'eaire connexe.
\end{rem}

\begin{proof}
Pour les espaces homog\`enes, le fait de poss\'eder un point ou un z\'ero-cycle de degr\'e 1 (local ou global) est invariant par stable birationalit\'e (utiliser par exemple \cite[Thm.~5.7]{Flo06} et le th\'eor\`eme de Lang-Nishimura) et il en va de m\^eme pour la densit\'e des points rationnels dans les points locaux. De m\^eme, le groupe de Brauer non ramifi\'e $\brnr X$ est un invariant birationnel stable et le quotient $\brnr X/\br k$ est toujours fini pour ces vari\'et\'es. Enfin, on sait bien que l'espace affine v\'erifie l'approximation faible. Ceci donne (i) pour toutes les propri\'et\'es de l'\'enonc\'e.

V\'erifier (ii) est facile pour les propri\'et\'es o\`u Brauer-Manin n'intervient pas. La fonctorialit\'e du groupe de Brauer donne les autres par un calcul direct.

Pour (iii), on sait d'apr\`es \cite[Prop.~3.4]{Borovoi96} que toute fibre du morphisme $X\to Y$ sur un point rationnel de $Y$ v\'erifie le principe de Hasse, l'approximation faible, et poss\`ede des points sur toute place finie. On obtient alors chacune des propri\'et\'es par la m\'ethode des fibrations.
\end{proof}

\begin{rem}
Sur un corps de nombres $k$ admettant une place r\'eelle, la propri\'et\'e ``avoir un point rationnel'' n'est pas une bonne propri\'et\'e des $k$-vari\'et\'es.
 
En revanche, sur un bon corps $K$ de dimension cohomologique $\leq 2$ et de caract\'eristique nulle (au sens de \cite[\S3.4]{BCTS08}), cette propri\'et\'e est bonne. Le th\'eor\`eme \ref{main thm} pour la propri\'et\'e ``avoir un point rationnel'' s'applique par exemple aux corps de nombres totalement imaginaires, aux corps $p$-adiques, aux corps des fractions d'anneaux int\`egres strictement hens\'eliens de dimension $2$ et de caract\'eristique r\'esiduelle nulle.
\end{rem}

\begin{proof}[Preuve du th\'eor\`eme \ref{main thm}]
On note $\rho : \mathcal{M} \to \TORS(G)$ la gerbe de Springer associ\'ee \`a $X$, $L$ le $K$-lien de $\mathcal{M}$, $\bar H$ le $\bar K$-groupe sous-jacent. On rappelle que l'on dispose des liens naturellement induits $L^\torf$ et $L^\f$ de groupes respectifs $\bar H^\torf$ et $\bar H^\f$, ainsi que les gerbes correspondantes $\mathcal{M}^\torf$ et $\mathcal{M}^\f$. Enfin, on note $E$ une extension de $\Gamma_K$ par $\bar H  (\bar K)$ associ\'ee \`a $\cal M$ et $E^\torf$, $E^\f$ les extension induites par les quotients $\bar H^\torf$, $\bar H^\f$ de $\bar H$.

On proc\`ede par r\'eductions successives.

\paragraph*{\'Etape 0 : On peut supposer $G^\ss$ simplement connexe.}
En effet, d'apr\`es \cite[Lem.~5.1]{Borovoi96}, il existe $K$-un groupe lin\'eaire connexe $G'$ avec $(G')^\ss$ simplement connexe et un morphisme surjectif $G'\to G$ \`a noyau torique. On peut alors regarder $X$ comme un espace homog\`ene de $G'$, ce qui nous permet de supposer que $G^\ss$ est simplement connexe dor\'enavant. Notons aussi que $\bar H^\f$ n'a pas \'et\'e modifi\'e dans le processus.

\paragraph*{\'Etape 1 : R\'eduction \`a un espace homog\`ene $X_1$ de $G_1=G^\tor\times\sl_{n,K}$ avec stabilisateur g\'eom\'etrique $\bar H_1=\bar H^\torf$.}
Posons $L_1:=L^\torf$, $\bar H_1:=\bar H^\torf$ et $\mathcal{M}_1:=\mathcal{M}^\torf$. Puisque le noyau $L^\ssu$ du morphisme $L \to L_1$ est repr\'esent\'e par le $\bar K$-groupe $\bar H^\ssu$, qui est caract\'eristique dans $\bar H$; et puisque $\bar H^\ssu$ est contenu dans $\bar G^\ssu$, on dispose d'un diagramme commutatif naturel de $K$-liens
\[
\xymatrix{
L  \ar@{->>}[r] \ar[d] & L_1 \ar[d] \\
\underline{\textup{lien}}(\bar G) \ar@{->>}[r] & \underline{\textup{lien}}(\bar G^\tor) \, .
}
\]
On en d\'eduit (cf.~\cite[IV.2.3.18]{Giraud}) un diagramme commutatif de morphismes de $K$-gerbes
\[
\xymatrix{
\mathcal{M}_X  \ar@{->>}[r] \ar[d] & \mathcal{M}_1 \ar[d] \\
\TORS(\bar G) \ar@{->>}[r] & \TORS(\bar G^\tor) \, .
}
\]

En outre, la proposition \ref{prop compatibilite des liens avec sln} assure qu'il existe un morphisme injectif de $K$-gerbes $\rho' : \mathcal{M}_1 \to \TORS(\sl_{n,K})$. On peut alors appliquer le th\'eor\`eme \ref{thm lemme sans nom} au diagramme pr\'ec\'edent et au morphisme $\rho'$ : il existe un $K$-espace homog\`ene $X_1$ de $G_1:=G^\tor \times \sl_{n,K}$ de gerbe de Springer isomorphe \`a $\mathcal{M}_1$; un $K$-espace homog\`ene $Y_1$ de $G\times \sl_{n,K}$ de gerbe de Springer $\mathcal{M}$; et des morphismes naturels de $K$-espaces homog\`enes 
\[\xymatrix{
& Y_1 \ar[rd]^q \ar[ld]_p & \\
X & & X_1 \, ,
}\]
de sorte que $p$ est un torseur sous $\sl_{n,K}$ et les fibres de $q$ sont des espaces homog\`enes de $G^\ssu$ \`a stabilisateurs g\'eom\'etriques $\bar H^\ssu$.

On voit alors que $Y_1$ est $K$-stablement birationnel \`a $X$, ce qui nous dit que $P(Y_1)\Rightarrow P(X)$. La description des fibres de $q$ nous permet ensuite de d\'eduire que $P(X_1)\Rightarrow P(Y_1)$, donc on peut se ramener \`a $X_1$.

\paragraph*{\'Etape 2 : R\'eduction \`a un espace homog\`ene $X_2$ de $\sl_{n,K}$ \`a stabilisateur $\bar H_2$, extension de $\bar H^\f$ par un $\bar K$-groupe de type multiplicatif et de gerbe de Springer $\mathcal{M}_2$ se surjectant naturellement sur $\mathcal{M}^\torf$.}
On rappelle que l'on dispose d'un espace homog\`ene $X_1$ de $G_1=G^\tor\times\sl_{n,K}$ \`a stabilisateur g\'eom\'etrique $\bar H_1=\bar H^\torf$ et de gerbe de Springer $\mathcal{M}_1=\mathcal{M}^\torf$.

D'apr\`es le lemme d'Ono, il existe un $K$-tore $R$ et une $K$-isog\'enie $Q\to G^\tor\times R$ avec $Q$ un $K$-tore quasi-trivial. Soit $A$ le noyau fini ab\'elien correspondant et consid\'erons alors le produit $G_2':=G_1\times R$ et l'isog\'enie $G_2:=\sl_{n,K}\times Q\to G_2'$, de noyau $A$. La vari\'et\'e $X_2':=X_1\times R$ est clairement un espace homog\`ene de $G_2'$ \`a stabilisateur g\'eom\'etrique $\bar H_1$, donc c'est aussi un espace homog\`ene de $G_2$ \`a stabilisateur $\bar H_2$, extension de $\bar H_1$ par $\bar A$. On note $\mathcal{M}_2'$ la gerbe de Springer de $X_2'$.

Or, notons que $\bar H_2$ est une extension de $\bar H^\f$ par un $\bar K$-groupe de type multiplicatif. En effet, le noyau de la projection naturelle $\bar H_2\to\bar H^\f$ est une extension de $\bar H^\tor$ par $\bar A$, qui sont tous les deux de type multiplicatif, le premier \'etant connexe (cf.~\cite[IV.\S1, Prop.~4.5]{DG}).\\

Comme la projection $X_2'\to X_1$ admet une section, on obtient que $P(X_2')\Rightarrow P(X_1)$, donc on peut se r\'eduire au cas de $X_2'$. De plus, puisque la projection est \'equivariante pour le morphisme $G_2\to G_2'\to G_1$, on dispose d'un morphisme surjectif naturel de $K$-gerbes $\mathcal{M}_2' \to \mathcal{M}_1 = \cal M^\torf$.

Enfin, puisque $G_2$ est un $K$-groupe sp\'ecial $K$-rationnel, il suffit d'appliquer le corollaire \ref{corollaire lemme sans nom} pour se ramener par $K$-stable birationalit\'e au cas d'un espace homog\`ene $X_2$ de $\sl_{n,K}$ \`a stabilisateur g\'eom\'etrique $\bar H_2$. La gerbe de Springer $\cal M_2$ de $X_2$ est isomorphe \`a $\cal M_2'$.

\paragraph*{\'Etape 3 : R\'eduction \`a un espace homog\`ene $X_3$ de $\sl_{n,K}$ \`a stabilisateur g\'eom\'etrique $\bar H_3$, extension de $\bar H^\f$ par un $\bar K$-groupe de type multiplicatif dont la composante neutre est $E^\f$-isog\`ene \`a un $E^\f$-tore quasi-trivial.} On dispose maintenant d'un espace homog\`ene $X_2$ de $G_2$, de gerbe $\mathcal{M}_2$ se surjectant sur $\mathcal{M}^\torf$, et \`a stabilisateur g\'eom\'etrique $\bar H_2$, extension de $\bar H^\f$ par un $\bar K$-groupe de type multiplicatif. En particulier, $\bar H_2^\f$ est une extension de $\bar H^\f$ par un $\bar K$-groupe fini ab\'elien qui agit trivialement sur $\bar H_2^\tor$ par conjugaison. Donc d'apr\`es la construction de la section \ref{section ono}, le tore $\bar H_2^\tor$ poss\`ede une structure naturelle de $E^\f$-tore.

Appliquons le lemme \ref{lemme Ono} \`a $\bar H_2^\tor$. On obtient un $E^\f$-tore $\bar R$ et une $E^\f$-isog\'enie entre $\bar R\times\bar H_2^\tor$ et un $E^\f$-tore quasi-trivial $\bar Q$. Si l'on note $E_2$ l'extension de $\Gamma_K$ par $\bar H_2(\bar K)$ correspondant \`a $\mathcal{M}_2$, on peut alors consid\'erer le produit semi-direct $\bar R(\bar K)\rtimes E_2$. Ce groupe s'ins\`ere naturellement dans une suite exacte
\[1\to \bar R(\bar K)\rtimes \bar H_2(\bar K) \to \bar R(\bar K)\rtimes E_2 \to \Gamma_K \to 1 \, ,\]
et on note $\mathcal{M}_3$ la gerbe associ\'ee. Le $\bar K$-groupe $\bar H_3 := R\rtimes \bar H_2$ est naturellement extension de $\bar H^\f$ par un $\bar K$-groupe de type multiplicatif et admet pour composante neutre le $E^\f$-tore $\bar R\times\bar H_2^\tor$, qui est $E^\f$-isog\`ene \`a $\bar Q$. C'est donc le stabilisateur g\'eom\'etrique souhait\'e.\\

Le diagramme commutatif \`a lignes exactes suivant (avec des fl\`eches $\sigma$ et $\pi$ d\'efinies de fa\c con \'evidente)
\[\xymatrix{
1 \ar[r] & \bar H_3 \ar@{->>}[d]_\pi \ar[r] & \bar R\rtimes E_2 \ar@{->>}[d]_\pi \ar[r] & \Gamma_K \ar@{=}[d] \ar[r] & 1 \\
1 \ar[r] & \bar H_2 \ar@{^{(}->}@/_1pc/[u]_\sigma \ar[r] & E_2 \ar@{^{(}->}@/_1pc/[u]_\sigma \ar[r] & \Gamma_K \ar[r] & 1,
}\]
correspond \`a des morphismes de $K$-gerbes $\sigma : \mathcal{M}_2\to \mathcal{M}_3$ et $\pi : \mathcal{M}_3\to \mathcal{M}_2$ tels que $\pi \circ \sigma = \id$. En particulier, $\mathcal{M}_3$ se surjecte sur $\mathcal{M}^\f$.

On choisit un morphisme injectif $\iota:\mathcal{M}_3 \to \TORS(\sl_{n,K})$ comme \`a la proposition \ref{prop compatibilite des liens avec sln}. Alors la proposition \ref{prop existence des espaces homogenes}, appliqu\'ee \`a
\[\xymatrix{
\mathcal{M}_2 \ar[d]_{\iota\circ\sigma} \ar[r]^{\sigma} & \mathcal{M}_3 \ar[d]^\iota \\
\TORS(\sl_{n,K}) \ar@{=}[r] & \TORS(\sl_{n,K}),
}\] assure l'existence d'un espace homog\`ene $X_2'$ (resp. $X_3$) de $\sl_{n,K}$ \`a stabilisateur g\'eom\'etrique $\bar H_2$ (resp. $\bar H_3$) et de gerbe de Springer $\mathcal{M}_2$ (resp. $\mathcal{M}_3$), et d'un $K$-morphisme d'espaces homog\`enes $X_2'\to X_3$. En outre le corollaire \ref{corollaire lemme sans nom} assure que $X_2'$ et $X_2$ sont stablement birationnels, donc $P(X_2') \Leftrightarrow P(X_2)$.

Notons que l'on dispose d\'esormais d'un diagramme commutatif de $\bar K$-groupes alg\'ebriques :
\[\xymatrix@R=3mm{
\bar H_2 \ar@{_{(}->}[dr]_\sigma \ar@{^{(}->}[drr]^s & & \\
& \bar H_3 \ar@{^{(}->}[r] & \sl_{n,\bar K}. \\
\bar R \ar@{^{(}->}[ur] \ar@{_{(}->}[urr]_i & &
}\]
Consid\'erons maintenant une copie $\sl_{n,K}'$ du m\^eme groupe $\sl_{n,K}$ avec des copies $s'$ et $i'$ des plongements de $\bar H_2$ et de $\bar R$ dans $\sl_{n,\bar K}'$. On fait agir $\sl_{n,K}$ sur $\sl_{n,K}'$ par conjugaison. Le produit semi-direct $\bar H_3=\bar R\rtimes \bar H_2$ se plonge alors dans le produit semi-direct $G_3:=\sl_{n,\bar K}'\rtimes\sl_{n,K}$ via $i'$ et $s$. En d'autres mots, on a le diagramme commutatif \`a lignes exactes scind\'ees suivant :
\[\xymatrix@R=5mm{
1 \ar[r] & \bar R \ar[r] \ar@{_{(}->}[d]_{i'} & \bar H_3 \ar[r] \ar@{_{(}->}[d]_{i'\rtimes s} & \bar H_2 \ar[r] \ar@{_{(}->}[d]^{s} \ar@/_1pc/[l] & 1 \\
1 \ar[r] & \sl_{n,\bar K}' \ar[r] & \bar G_3 \ar[r] & \sl_{n,\bar K} \ar[r] \ar@/_1pc/[l] & 1.
}\]

Puisque l'extension $E_3$ correspond au produit semi-direct $R\rtimes E_2$ et que c'est ainsi qu'on obtient $\sigma : \mathcal{M}_2\to \mathcal{M}_3$ et $\pi : \mathcal{M}_3\to \mathcal{M}_2$, on d\'eduit de ce diagramme un diagramme commutatif de $K$-gerbes
\[
\xymatrix{
\mathcal{M}_3 \ar@{->>}[r]^\pi \ar@{^{(}->}[d] & \mathcal{M}_2 \ar@{^{(}->}[d] \ar@/_1pc/[l]_{\sigma} \\
\TORS(G_3) \ar@{->>}[r] & \TORS(\sl_{n,K}) \ar@/_1pc/[l] \, ,
}
\]

Par la proposition \ref{prop existence des espaces homogenes}, ce diagramme induit un morphisme surjectif d'espaces homog\`enes $\pi : X_3' \to X_2'$, o\`u $X_3'$ est un $K$-espace homog\`ene de $G_3$ et de gerbe $\mathcal{M}_3$. On en d\'eduit aussi que ce morphisme est scind\'e par un morphisme $\sigma : X_2' \to X_3'$.
Par cons\'equent, on sait que $P(X_3') \Rightarrow P(X_2')$. En outre, le groupe $G_3$ est sp\'ecial et $K$-rationnel, donc le corollaire \ref{corollaire lemme sans nom} assure que $X_3'$ est stablement birationnel \`a l'espace homog\`ene $X_3$ de $\sl_{n,K}$ dont la gerbe est $\mathcal{M}_3$, donc $P(X_3') \Leftrightarrow P(X_3)$.

Finalement, on a bien $P(X_3) \Rightarrow P(X_2)$, donc	on s'est ramen\'e \`a un espace homog\`ene $X_3$ de $\sl_{n,K}$ \`a stabilisateur $\bar H_3$ et gerbe de Springer $\mathcal{M}_3$ se surjectant sur $\mathcal{M}^\f$.

\paragraph*{\'Etape 4 : R\'eduction \`a un espace homog\`ene $X_4$ de $\sl_{n,K}$ \`a stabilisateur g\'eom\'etrique $\bar H_4$, extension de $\bar H^\f$ par un $\bar K$-groupe fini ab\'elien.} On rappelle que l'on dispose d'un espace homog\`ene $X_3$ de gerbe $\mathcal{M}_3$ se surjectant sur $\mathcal{M}^\f$, et dont le stabilisateur g\'eom\'etrique $\bar H_3$ est extension de $\bar H^\f$ par un groupe de type multiplicatif dont la composante neutre $\bar H_3^\tor$ est un $E^\f$-tore $E^\f$-isog\`ene \`a un $E^\f$-tore quasi-trivial $\bar Q$. Rappelons que $E_3$ (resp.~$E_3^\f$) est une extension de $\Gamma_K$ par $\bar H_3(\bar K)$ (resp.~$\bar H_3^\f$) d\'efinissant $\mathcal{M}_3$ (resp.~$\mathcal{M}_3^\f$). La construction du $E_3$-tore $\bar H_3^\tor$ (cf.~la section \ref{section ono}) nous fournit la suite exacte
\[1\to \bar H_3^\tor(\bar K)\to E_3\to E_3^\f\to 1 \, .\]
Puisque $E_3^\f$ est un groupe profini, cette extension correspond \`a une classe $\zeta\in H^2(E_3^\f,\bar H_3^\tor(\bar K))$. Or ce groupe est de torsion (car $E_3^\f$ est profini), donc il existe un entier $m \geq 1$ tel que $m\zeta=0$. Cela veut dire que,  si l'on note par $\bar A'$ le $\bar K$-sous-groupe de $m$-torsion de $\bar H_3^\tor$, la suite exacte longue de cohomologie nous permet de d\'eduire que l'extension ci-dessus provient d'une extension de $E_3^\f$ par $\bar A'(\bar K)$. Autrement dit, on a un diagramme commutatif \`a lignes exactes
\[\xymatrix{
1 \ar[r] & \bar A'(\bar K) \ar[r] \ar[d] & E_4' \ar[r] \ar[d] & E_3^\f \ar[r] \ar@{=}[d] & 1 \\
1 \ar[r] & \bar H_3^\tor(\bar K) \ar[r] & E_3 \ar[r] & E_3^\f \ar[r] & 1 \, .
}\]
On d\'efinit alors $\bar A$ comme le noyau du morphisme compos\'e de $E^\f$-tores
\[\rho:\bar H_3^\tor\xrightarrow{\times m}\bar H_3^\tor \to \bar Q \, .\]
Puisque $\bar A'\subset \bar A\subset \bar H_3^\tor$, l'extension $E_4'$ induit une extension $E_4$ de $E_3^\f$ par $\bar A(\bar K)$ qui factorise le diagramme pr\'ec\'edent. Or, on dispose d'un morphisme naturel $E_4\to\Gamma_K$ via son quotient $E_3^\f$ et l'on d\'efinit alors $\bar H_4$ comme le $\bar K$-groupe fini correspondant au noyau (abstrait) de ce morphisme. On obtient alors la suite exacte sup\'erieure dans le diagramme commutatif suivant
\[\xymatrix{
1 \ar[r] & \bar H_4(\bar K) \ar[r] \ar[d] & E_4 \ar[r] \ar[d] & \Gamma_K \ar[r] \ar@{=}[d] & 1 \\
1 \ar[r] & \bar H_3(\bar K) \ar[r] & E_3 \ar[r] & \Gamma_K \ar[r] & 1 \, .
}\]
On interpr\`ete ce dernier diagramme comme un morphisme injectif de $K$-gerbes $\mathcal{M}_4 \to \mathcal{M}_3$. En utilisant le morphisme injectif $\mathcal{M}_3 \to \TORS(\sl_{n,K})$ donn\'e par l'espace homog\`ene $X_3$, la proposition \ref{prop existence des espaces homogenes} assure donc l'existence d'un espace homog\`ene $X_4$ de $\sl_{n,K}$ de gerbe $\mathcal{M}_4$, donc de stabilisateur g\'eom\'etrique $\bar H_4$, muni d'un morphisme surjectif $X_4\to X_3$ \'equivariant pour les actions de $\sl_{n,K}$.

Notons enfin que le groupe $\bar H_4$ est une extension $\bar H^\f$ par un $\bar K$-groupe ab\'elien. En effet, le noyau du morphisme compos\'e $\bar H_4\to\bar H_3\to\bar H^\f$ est contenu dans un $\bar K$-groupe de type multiplicatif par construction de $\bar H_3$. Et le morphisme est clairement surjectif car $\bar H_4$ se surjecte sur $\bar H_3^\f$ par construction.\\

On affirme que l'espace $X_4$ est $K$-stablement birationnel \`a $X_3$, ce qui suffit pour conclure la d\'emonstration du th\'eor\`eme puisqu'alors $P(X_4) \Leftrightarrow P(X_3)$. Pour d\'emontrer cela, il suffira de d\'emontrer que, si l'on note $\xi$ le point g\'en\'erique de $X_3$ et $K(\xi)$ le corps de fonctions correspondant, alors la fibre g\'en\'erique $X_{4,\xi}$ est $K(\xi)$-rationnelle.

Le point g\'en\'erique $\xi$ de $X_3$ nous donne une $K(\xi)$-forme naturelle $H_{3,\xi}$ de $\bar H_3$ et donc une $K(\xi)$-forme $H^\f_{3,\xi}$ de $\bar H_3^\f$. On en d\'eduit une section $\Gamma_{K(\xi)}\to E_{3,\xi}^\f$ de l'extension associ\'ee \`a $\mathcal{M}_{3,\xi}^\f$ qui d\'efinit l'action de $\Gamma_{K(\xi)}$ sur $\bar H_3^\f$. Or $E^\f$ est un quotient de $E_3^\f$ dont le noyau correspond au noyau de la projection $\bar H_3^\f\to\bar H^\f$. Donc ce noyau est stable par l'action de $E_3^\f$ par conjugaison, et donc stable par l'action de $E_{3,\xi}^\f$ et par celle de $\Gamma_{K(\xi)}$. En d'autres mots, ce noyau est un $K(\xi)$-sous-groupe distingu\'e de $H_{3,\xi}^\f$ et l'on obtient ainsi une $K(\xi)$-forme $H^\f_\xi$ de $\bar H^\f$. La $K(\xi)$-gerbe $\mathcal{M}_{\xi}^\f$ est donc isomorphe \`a $\TORS(H^\f_\xi)$ et on obtient une section $s_\xi:\Gamma_{K(\xi)}\to E_\xi^\f$ de l'extension correspondant \`a $\mathcal{M}_\xi^\f$.

Notons en outre que la fonctorialit\'e de la construction des $E^\f$-tores fait de tout $E^\f$-tore un $E_{\xi}^\f$-tore (via la projection $E^\f_\xi\to E^\f$). En particulier, puisque l'on dispose de la section $s_\xi$ de $E^\f_\xi$, tout $E^\f$-tore admet une $K(\xi)$-forme naturelle (pour $\bar H_3^\tor$, c'est justement $H_{3,\xi}^\tor$). On en d\'eduit que $\bar Q$ admet une $K(\xi)$-forme $Q_\xi$ qui est quasi-triviale en tant que $K(\xi)$-tore et le morphisme $\rho : H_{3,\xi}^\tor \to Q_\xi$ est alors un $K(\xi)$-morphisme. D'o\`u une $K(\xi)$-forme naturelle $A_\xi$ de $\bar A$ et une suite exacte de $K(\xi)$-groupes ab\'eliens
\[1\to A_\xi \to H_{3,\xi}^\tor \xrightarrow{\rho} Q_\xi \to 1 \, .\]

Or la vari\'et\'e $X_{4,\xi}$ correspond \`a la fibre du morphisme $X_{4,K(\xi)}\to X_{3,K(\xi)}$ au $K(\xi)$-point $\xi$. Le groupe $\sl_{n,K(\xi)}$ agit sur les deux vari\'et\'es de fa\c con compatible, donc le groupe $H_{3,\xi}$ le fait aussi en tant que sous-$K(\xi)$-groupe de $\sl_{n,K(\xi)}$. Mais puisque ce sous-groupe fixe le point $\xi$, il agit naturellement par restriction sur la fibre $X_{4,\xi}$. Cette action est transitive sur $\overline{K(\xi)}$ et a pour stabilisateur $\bar H_{4,\overline{K(\xi)}}$. Le tore $H_{3,\xi}^\tor$ agit alors sur $X_{4,\xi}$ par restriction et l'on v\'erifie ais\'ement que cette action est aussi transitive sur $\overline{K(\xi)}$ puisque $\bar H_{4,\overline{K(\xi)}}$ rencontre toutes les composantes connexes de $H_{3,\overline{K(\xi)}}$. Le stabilisateur g\'eom\'etrique de cette action est clairement $H_{3,\overline{K(\xi)}}^\tor\cap \bar H_{4,\overline{K(\xi)}}=\bar A_{\overline{K(\xi)}}$, qui est d\'efini sur $K(\xi)$ comme on l'a d\'ej\`a d\'emontr\'e. La vari\'et\'e $X_{4,\xi}$ est alors un \emph{torseur} sous le $K(\xi)$-groupe quotient $H_{3,\xi}^\tor/A_\xi=Q_\xi$. Or ce tore est quasi-trivial, donc $K(\xi)$-rationnel, et $H^1(K(\xi),Q_\xi)=0$. Par cons\'equent, $X_{4,\xi}$ est isomorphe \`a $Q_\xi$, et donc est $K(\xi)$-rationnelle.

En posant $\bar F=\bar H_4$ et $X'=X_4$, ceci conclut la preuve du th\'eor\`eme.
\end{proof}

\begin{rem}
En suivant la preuve, on peut constater que non seulement $\bar F$ est une extension de $\bar H^\f$ par un $\bar K$-groupe ab\'elien, mais en fait on a un morphisme naturel et surjectif de $K$-liens $L_{X'}\to L^\f$ qui envoie $\cal M_{X'}$ en $\cal M^\f$. De plus, lorsque $X$ poss\`ede un $K$-point, il est facile de voir que cette construction fournit un espace homog\`ene $X'$ qui poss\`ede aussi un $K$-point et le morphisme $\bar F\to \bar H^\f$ est alors d\'efini sur $K$ (une autre fa\c con de se convaincre de ce fait est de suivre la preuve dans \cite{GLA-BMWA}). Ce fait devrait \^etre utile pour obtenir des r\'esultats inconditionnels sur \eqref{BMPH} pour des stabilisateurs non connexes \`a partir de r\'esultats positifs pour des stabilisateurs finis, suivant l'exemple de \cite[\S5]{GLA-BMWA}.
\end{rem}

\bigskip

\noindent
Cyril \textsc{Demarche}
\newline Sorbonne Universit\'es, UPMC Univ Paris 06 
\newline Institut de Math\'ematiques de Jussieu-Paris Rive Gauche
\newline UMR 7586, CNRS, Univ Paris Diderot, Sorbonne Paris Cit\'e, F-75005, Paris, France  
\newline and
\newline D\'epartement de math\'ematiques et applications,
\newline \'Ecole normale sup\'erieure, CNRS, PSL Research University,
\newline 45 rue d'Ulm, 75005 Paris, France

\smallskip

\noindent cyril.demarche@imj-prg.fr

\bigskip

\noindent
Giancarlo \textsc{Lucchini Arteche}
\newline Departamento de Matem\'aticas
\newline Facultad de Ciencias
\newline Universidad de Chile
\newline Las Palmeras 3425, \~Nu\~noa, Santiago, Chile

\smallskip

\noindent luco@uchile.cl

\end{document}